\documentclass[12pt,a4paper]{article}
\usepackage{amssymb, amscd, amsthm, amsmath,latexsym, amstext,mathrsfs}
\usepackage[all]{xy}
\usepackage[dvips]{graphicx}

\numberwithin{equation}{section}

\newtheorem{thm}{Theorem}[section]
\newtheorem{coro}[thm]{Corollary}
\newtheorem{lemma}[thm]{Lemma}
\newtheorem{prop}[thm]{Proposition}

\theoremstyle{remark}
\newtheorem{remark}[thm]{\textbf{Remark}}

\theoremstyle{definition}
\newtheorem{defn}[thm]{Definition}
\newtheorem{example}[thm]{Example}
\newtheorem{question}[thm]{Question}

\newcommand{\red}{\mathrm{red}}
\newcommand{\Br}{\mathrm{Br}}

\newcommand{\ram}{\mathrm{ram}}
\newcommand{\Ram}{\mathrm{Ram}}
\newcommand{\di}{\mathrm{div}}
\newcommand{\Pic}{\mathrm{Pic}}
\newcommand{\et}{\text{\'et}}
\newcommand{\Spec}{\mathrm{Spec}\,}
\newcommand{\Proj}{\mathrm{Proj}\,}
\newcommand{\Hom}{\mathrm{Hom}}
\newcommand{\lra}{\longrightarrow}
\newcommand{\empha}[1]{\textbf{\emph{#1}}}
\newcommand{\ov}[1]{\overline{#1}}
\newcommand{\set}[1]{\{\,{#1}\,\}}
\newcommand{\simto}{\xrightarrow{\sim}}

\begin{document}
\title{\textbf{Division Algebras and Quadratic Forms over Fraction Fields of Two-dimensional
Henselian Domains}}
\author{Yong HU\footnote{Math\'{e}matiques, B\^{a}timent 425, Universit\'{e} Paris-Sud, 91405,
Orsay Cedex,  France,\ \ \ e-mail: hu1983yong@gmail.com}}

\maketitle

\begin{abstract}
Let $K$ be the fraction field of a 2-dimensional, henselian, excellent local domain with
finite residue field $k$. When the characteristic of $k$ is not 2, we prove that every
quadratic form of rank $\ge 9$ is isotropic over $K$ using methods of Parimala and
Suresh, and we obtain the local-global principle for isotropy of quadratic forms of rank 5
with respect to discrete valuations of $K$. The latter result is proved by making a careful study
of ramification and cyclicity of division algebras over the field
$K$, following Saltman's methods. A key step is the proof of the following result, which answers a question of Colliot-Th\'el\`ene--Ojanguren--Parimala: For a Brauer class over $K$ of prime order $q$ different from the characteristic of $k$, if it is cyclic of degree
$q$ over the completed field $K_v$ for every discrete valuation $v$
of $K$, then the same holds over $K$.  This local-global principle for cyclicity is also established over function fields of
$p$-adic curves with the same method.
\end{abstract}

\section{Introduction}\label{sec1}

Division algebras and quadratic forms over a field have been objects of interest
in classical and modern theories of algebra and number theory. They may also be naturally and closely related to
the study of semisimple algebraic groups of classical types. In recent years, there has been much interest in problems on
division algebras and quadratic forms over function fields of two-dimensional integral schemes (which we call \emph{surfaces}).

Mostly, surfaces that have been studied are those equipped with a dominant quasi-projective morphism
to the spectrum of a normal henselian excellent local domain $A$. If $A$ is of (Krull) dimension 0, these are algebraic surfaces over a field. Over function fields of these surfaces, de Jong (\cite{deJ04}) and Lieblich  (\cite{Lie11}) have proven remarkable theorems concerning the period-index problem.
If $A$ is of dimension 1, the surfaces of interest are  called \emph{arithmetic surfaces} by some authors.
Over function fields of arithmetic surfaces,  several different methods have been developed, by Saltman,
Lieblich \cite{Lie07},  Harbater--Hartmann--Krashen (\cite{HHK}) and others,  to study division algebras and/or quadratic forms.
Among others, the methods pioneered by Saltman in a series of papers (\cite{Salt97}, \cite{Salt07},
\cite{Salt08}) have been important ingredients in several work by others, including Parimala and Suresh's proof (cf. \cite{PS10} and \cite{PS10b}) of the fact that over
a non-dyadic $p$-adic function field every quadratic form of dimension $\ge 9$ has a nontrivial zero. In contrast with the arithmetic case, it seems that in the case where $A$ is $2$-dimensional, fewer results have been established in earlier work.

\

In this paper, we concentrate on the study of  division algebras and quadratic forms over the function field $K$ of a surface that admits a proper birational morphism to the spectrum of a 2-dimensional henselian excellent local domain $R$. The spectrum $\Spec R$ will sometimes be called a \emph{local henselian surface} and a regular surface $X$ equipped with a proper birational morphism $X\to \Spec R$ will be referred to as a \emph{regular proper model} of $\Spec R$.  As typical examples, one may take $R$ to be the henselization at a closed point of an algebraic or an arithmetic surface, or the integral closure
of the ring $A[\![t]\!]$ of formal power series  in a finite extension of its fraction field $\mathrm{Frac}(A[\![t]\!])$, where $A$ is a complete discrete valuation ring. Note that the ring $R$ need not be regular in our context.

Let $k$ denote the residue field of $R$. When $k$ is separably closed, many problems over the function field $K$ (e.g. period-index, cyclicity of division algebras, $u$-invariant and local-global principle for quadratic forms of lower dimension) have been solved by Colliot-Th\'el\`ene--Ojanguren--Parimala
in \cite{CTOP}. In the case with $k$  finite, only the local-global principle for quadratic forms of rank 3 or 4 is proved in that paper. In a more recent work of
Harbater--Hartmann--Krashen (\cite{HHK11b}), some results are obtained with less restrictive assumptions on the residue field but more restrictions on the shape of the ring $R$.

\

While the proofs pass through many analyses on ramification of division algebras, our primary goals are the following two theorems on quadratic forms.

\begin{thm}\label{thm1p4temp}
Let $R$ be a $2$-dimensional, henselian, excellent local domain with
finite residue field $k$ and fraction field $K$. Assume that $2$ is invertible in $k$.
Let $\Omega_R$ be the set of discrete valuations of $K$ which correspond
to codimension $1$ points of regular proper models  of $\Spec R$.

Then quadratic forms of rank $5$ over $K$ satisfy the local-global
principle with respect to discrete valuations in $\Omega_R$, namely,
if a quadratic form $\phi$ of rank $5$ over $K$ has a nontrivial
zero over the completed field $K_v$ for every $v\in\Omega_R$, then
$\phi$ has nontrivial zero over $K$.
\end{thm}

The following theorem amounts to saying that the field $K$ has $u$-invariant (cf. $\S$\ref{sec4p2}) equal to 8.

\begin{thm}\label{thm1p5temp}
Let $R$ be a $2$-dimensional, henselian, excellent local domain with finite residue field $k$ and fraction field $K$.
Assume that $2$ is invertible in $k$.

Then every quadratic form of rank $\ge 9$ has a nontrivial zero over $K$.
\end{thm}

 Over the function field of an arithmetic surface over a complete discrete valuation ring,
the same local-global principle as in Thm.$\;$\ref{thm1p4temp} is proved for
all quadratic forms of rank $\ge 3$ in \cite[Thm.$\;$3.1]{CTPaSu}, by using the patching method of  \cite{HHK}. In the case that $R=A[\![t]\!]$ is a ring of formal power series in
one variable over a complete discrete valuation ring  $A$, the same
type of local-global principle has been proven for quadratic forms
of rank $\ge 5$ in \cite{Hu10}, using the arithmetic case established by Colliot-Th\'el\`ene--Parimala--Suresh (\cite{CTPaSu}). (See Remark$\;$\ref{remark4p2temp} for more
information.) However, in the general local henselian case, the lack of an appropriate patching method has been an obstacle to proving the parallel
local-global result. So for a field $K$ as in Thm.$\;$\ref{thm1p4temp}, the local-global principle for quadratic forms of rank 6, 7 or 8 remains open.

In the case of a $p$-adic function field, it is known that at least three methods
can be used to determine the $u$-invariant: the cohomological method of Parimala--Suresh (\cite{PS10}), the patching method of Harbater--Hartmann--Krashen
(\cite{HHK}) and the method of Leep (\cite{Le10}), which is built on results from \cite{HB10}. But in the case of the function field of a local henselian surface considered here, not all of them seem to still work. For the fraction field of a power series ring $R=A[\![t]\!]$ over a complete discrete valuation ring with finite residue field, it is known that the $u$-invariant is $\le 8$ (cf. \cite[Coro.$\;$4.19]{HHK}). Our proof of this result for general $R$ (with finite residue field) follows the method of Parimala and Suresh (\cite{PS10}, \cite{PS10b}).

Theorem$\;$\ref{thm1p5temp} implies that the $u$-invariants $u(K)$ of the fraction field $K$ and $u(k)$ of the residue field $k$ satisfy
the relation $u(K)=4u(k)$ when the residue field $k$ is finite. A question of Suresh asks if this relation still holds when
$k$ is an arbitrary field of characteristic $\neq 2$. The answer is known to be affirmative in some other special cases, but the
general case seems to remain open.  (See Question$\;$\ref{Ques4p11temp} for more information.)

\

As a byproduct, we also obtain (under the assumption of
Thm.$\;$\ref{thm1p5temp}) a local-global principle for torsors of the
special orthogonal group $\mathrm{SO}(\phi)$ of a quadratic form
$\phi$ of rank $\ge 2$ over $K$ (cf. Thm.$\;$\ref{thm4p11temp}). In fact, Theorem$\;$\ref{thm1p5temp} will also be useful
in the study of local-global principle for torsors under some simply connected groups of classical types over $K$ (cf. \cite{Hu12}).

\

The main tools we will need to prove Theorem$\;$\ref{thm1p4temp} come from technical analyses of ramification behaviors of division algebras using
methods developed by Saltman (\cite{Salt97}, \cite{Salt07}, \cite{Salt08}). A key ingredient is the following result.

\begin{thm}\label{thm1p3temp}
Let $R$ be a $2$-dimensional, henselian, excellent local
domain with finite residue field $k$, $q$
 a prime number unequal to the characteristic of  $k$, $K$ the fraction field of $R$ and
$\alpha\in\Br(K)$ a Brauer class of order $q$. Let $\Omega_R$ be the
set of discrete valuations of $K$ which correspond to codimension $1$ points of regular proper models  of $\Spec R$.

If for every $v\in \Omega_R$, the Brauer class $\alpha\otimes_KK_v\in \Br(K_v)$ is
 represented by a cyclic algebra of degree $q$ over the completed field $K_v$,
then $\alpha$ is represented by a cyclic algebra of degree $q$ over $K$.
\end{thm}
Actually, as the same proof applies to the function field of a $p$-adic curve,
a similar result over $p$-adic function fields, which seems not to have been treated in the literature, holds as well (cf. Thm.$\;$\ref{thm3p23temp}).
Note that a special case of Thm.$\;$\ref{thm1p3temp} answers a question raised in \cite[Remark$\;$3.7]{CTOP}.

\

Here is a brief description of the organization of the paper, together with some auxiliary results obtained in the process of proving the above mentioned theorems.

Section \ref{sec2} is concerned with preliminary reviews on
Brauer groups and  Galois symbols. The goal is to introduce some basic notions and recall standard results that we will frequently use later. In section \ref{sec3}
we recall some of the most useful techniques and results from Saltman's papers (\cite{Salt97}, \cite{Salt07}, \cite{Salt08}) and we prove Theorem$\;$\ref{thm1p3temp}. We also prove over the field $K$ considered in Thm.$\;$\ref{thm1p4temp} two local analogs of more global results Saltman had shown: that the index of
a Brauer class of period prime to the residue characteristic divides the square of its period, and that a class of prime index $q$ which is different from the residue characteristic is represented by a cyclic algebra of degree $q$. This last statement is proved by generalizing a result of Saltman on
 modified Picard groups. Finally, we will concentrate on results about quadratic forms in section \ref{sec4}. The proofs of Theorems \ref{thm1p4temp} and \ref{thm1p5temp} build upon the work of Parimala and Suresh (\cite{PS10}, \cite{PS10b}) as well as a result from Saito's class
field theory for 2-dimensional local rings (\cite{Sai87}).

\

To ease the discussions,  we fix some notations and terminological
conventions for all the rest of the paper:

\

$\bullet$ All schemes are assumed to be noetherian and separated.
All rings under consideration will be noetherian (commutative with
1).

$\bullet$ A \emph{curve} (resp. \emph{surface}) means an
integral scheme of dimension 1 (resp. 2).

$\bullet$ Given a scheme $X$, we denote by $\Br(X)$ its
cohomological Brauer group, i.e., $\Br(X):=H^2_{\et}(X\,,\,\mathbb{G}_m)$.
If $X=\Spec A$ is affine, we write $\Br(A)$ instead of $\Br(\Spec
A)$.

$\bullet$ If $X$ is a scheme and $x\in X$, we write $\kappa(x)$ for
the residue field of $x$, and if $Z\subseteq X$ is an irreducible
closed subset with generic point $\eta$, then we write
$\kappa(Z):=\kappa(\eta)$.

$\bullet$ The reduced closed subscheme of a scheme $X$ will be written as $X_{\red}$.

$\bullet$ A discrete valuation will always be assumed normalized
(nontrivial) and  of rank 1.

$\bullet$ Given a field $F$ and a scheme $X$ together with a
morphism $\Spec F\to X$, $\Omega(F/X)$ will denote the set of
discrete valuations of $F$ which have a center on $X$. If $X=\Spec
A$ is affine, we write $\Omega(F/A)$ instead of $\Omega(F/\Spec A)$.

$\bullet$ Given a scheme $X$ and $i\in\mathbb{N}$, $X^{(i)}$ will denote the set of codimension $i$ points of $X$, i.e., $X^{(i)}=\set{x\in X\,|\,\dim\mathscr{O}_{X,\,x}=i}$.
If $X$ is a normal integral scheme with function field $F$, we will sometimes identify $X^{(1)}$ with the set of discrete valuations of $F$ corresponding to points in $X^{(1)}$.

$\bullet$ For an abelian group $A$ and a positive integer
$n$, let $A[n]$ denote the subgroup consisting of $n$-torsion
elements of $A$ and let $A/n=A/nA$, so that there is a natural exact sequence
\[
0\lra A[n]\lra A\overset{n}{\lra}A\lra A/n\lra 0\,.
\]

$\bullet$ Given a field $F$, $F_s$ will denote a fixed separable closure of $F$ and $G_F:=\mathrm{Gal}(F_s/F)$ the absolute Galois
group. Galois cohomology $H^i(G_F\,,\,\cdot)$ of the group $G_F$ will be written $H^i(F\,,\,\cdot)$ instead.

$\bullet$ $R$ will always denote a 2-dimensional, henselian, excellent local domain with fraction field $K$ and residue field $k$.

$\bullet$ By a \emph{regular proper model} of $\Spec R$, we mean a regular integral scheme $\mathcal{X}$ equipped with a proper birational morphism $\mathcal{X}\to\Spec R$. A discrete valuation of $K$ which corresponds to a codimension 1 point of a regular proper model of $\Spec R$ will be referred to as
a \emph{divisorial valuation} of $K$. We denote by $\Omega_R$ the set of divisorial valuations of $K$.

\section{Some preliminaries}\label{sec2}

\subsection{Brauer groups of low dimensional schemes}\label{sec2p2}
Since we will frequently use arguments related to Brauer groups of curves or surfaces, let us briefly review some basic facts in this respect.

\begin{thm}[{\cite{GB2}, \cite{CTOP}}]\label{thm2p1temp}
Let $X$ be a $($noetherian$)$ scheme of dimension $d$.

$(\mathrm{i})$ If $d\le 1$, then the natural map $\Br(X)\to\Br(X_{\red})$ is an isomorphism.

$(\mathrm{ii})$ If $X$ is regular and integral with function field $F$, then the natural map $\Br(X)\to\Br(F)$ is injective.

$(\mathrm{iii})$ If $X$ is regular, integral with function field $F$ and of dimension $d\le 2$, then $\Br(X)=\bigcap_{x\in X^{(1)}}\Br(\mathscr{O}_{X,\,x})$ inside $\Br(F)$.

$(\mathrm{iv})$ Let $A$ be a henselian local ring and let $X\to\Spec
A$ be a proper morphism whose closed fiber $X_0$ has dimension $\le
1$. If $X$ is regular and of dimension $2$, then the natural map
$\Br(X)\to\Br(X_0)$ is an isomorphism.
\end{thm}
\begin{proof}
(i) \cite[Lemma$\;$1.6]{CTOP}. (ii) \cite[Coro.$\;$1.8]{GB2}. (iii)
\cite[Coro.$\;$2.2 and Prop.$\;$2.3]{GB2}. (iv) \cite[Thm.$\;$1.8
(c)]{CTOP}.
\end{proof}

The following property for fields, already considered in
\cite{Salt97}, will be of interest to us:

\begin{defn}\label{defn2p2temp}
We say a field $k$ has \empha{property $B_1$} or $k$ is a $B_1$
\empha{field}, if for every proper regular integral $($not
necessarily geometrically integral$)$ curve $C$ over the field $k$,
one has $\Br(C)=0$.
\end{defn}

\begin{example}\label{exam2p3temp}
Here are some examples of $B_1$ fields:

(1) A separably closed field $k$ has property $B_1$
(\cite[Coro.$\;$5.8]{GB3}).

(2) A finite field $k$ has property $B_1$. This is classical by
class field theory, see also \cite[p.97]{GB3}.

(3) If $k$ has property $B_1$, then so does any algebraic field
extension $k'$ of $k$.
\end{example}

\begin{prop}\label{prop2p4temp}
Let $k$ be a $B_1$ field.

$(\mathrm{i})$ For any proper $k$-scheme $X$ of dimension $\le 1$, one has
$\Br(X)=0$.

$(\mathrm{ii})$ The cohomological dimension $\mathrm{cd}(k)$ of $k$ is $\le 1$, i.e.,
for every torsion $G_k$-module $A$, $H^i(k\,,\,A)=0$ for all $i\ge 2$.

$(\mathrm{iii})$ If the characteristic of $k$ is not $2$, then every quadratic form
of rank $\ge 3$ has a nontrivial zero over $k$.
\end{prop}
\begin{proof}
(i) By Thm.$\;$\ref{thm2p1temp} (i), we may assume $X$ is reduced.

For the 0-dimensional case, it suffices to prove that $\Br(L)=0$ for
a finite extension field $L$ of $k$. Indeed, the $B_1$ property
implies that $\Br(\mathbb{P}_L^1)=0$. The existence of $L$-rational points
on $\mathbb{P}^1_L$ shows that the natural map $\Br(L)\to\Br(\mathbb{P}^1_L)$
induced by the structural morphism $\mathbb{P}^1_L\to\Spec L$ is injective.
Hence, $\Br(L)=0$.

Now assume that $X$ is reduced of dimension 1. Let $X'\to X$ be the
normalization of $X$. By \cite[Prop.$\;$1.14]{CTOP}, there is a
0-dimensional closed subscheme $D$ of $X$ such that the natural map
$\Br(X)\to\Br(X')\times\Br(D)$ is injective. Now $X'$ is a disjoint
union of finitely many proper regular $k$-curves, so $\Br(X')=0$ by
the $B_1$ property. We have $\Br(D)=0$ by the 0-dimensional case,
whence $\Br(X)=0$ as desired.

(ii) As a special case of (i), we have $\Br(k')=0$ for every finite
separable extension field $k'$ of $k$. This implies $\mathrm{cd}(k)\le 1$ by \cite[p.88, Prop. 5]{Ser2}.

(iii) By (ii), we have in particular $\Br(k)[2]=H^2(k\,,\,\mu_2)=0$. Thus
every quaternion algebra over $k$ is split and the associated
quadric has a $k$-rational point. Up to a scalar multiple, every
nonsingular 3-dimensional quadratic form is associated to a quaternion algebra, and hence
isotropic.
\end{proof}

The following corollary is essentially proven in  \cite[Corollaries$\;$1.10 and 1.11]{CTOP}.

\begin{coro}\label{coro2p5temp}
Let $A$ be a $($noetherian$)$ henselian local ring and let
$X\to\Spec A$ be a proper morphism whose closed fiber $X_0$ is of
dimension $\le 1$. Assume that the residue field of $A$ has property
$B_1$.

If $X$ is regular and of dimension $2$, then $\Br(X)=0$.
\end{coro}
\begin{proof}
Combine Thm.$\;$\ref{thm2p1temp} (iv) and Prop.$\;$\ref{prop2p4temp} (i).
\end{proof}

\subsection{Symbols and unramified cohomology}\label{sec2p1}
This subsection is devoted to a quick review of a few standard facts about Galois symbols and residue maps. For more information, we refer the
reader to \cite{CT95}.

\

Let $F$ be a field  and $v$ a discrete valuation of $F$ with valuation ring $\mathcal{O}_v$ and residue field $\kappa(v)$. Let $n>0$ be a positive integer unequal
to the characteristic of $\kappa(v)$. Let $\mu_n$ be the Galois module on the group of $n$-th roots of unity. For an integer $j\ge 1$, let $\mu_n^{\otimes j}$
denote the Galois module given by the tensor product of $j$ copies of $\mu_n$ and define\[
\mu_n^{\otimes 0}:=\mathbb{Z}/n\quad\text{and }\quad \mu_n^{\otimes (-j)}:=\Hom(\mu_n^{\otimes j}\,,\,\mathbb{Z}/n)\,,
\]where as usual $\mathbb{Z}/n$ is regarded as a trivial Galois module. Kummer theory gives a canonical isomorphism $H^1(F\,,\,\mu_n)\cong F^*/F^{*n}$. For an
element $a\in F^*$, we denote by $(a)$ its canonical image in $H^1(F\,,\,\mu_n)=F^*/F^{*n}$. For $\alpha\in H^i(F\,,\,\mu_n^{\otimes j})$, the cup product
$\alpha\cup (a)\in H^{i+1}(F\,,\,\mu_n^{\otimes (j+1)})$ will be simply written as $(\alpha\,,\,a)$. In particular, if $a_1\,,\dotsc, a_i\in F^*$, $(a_1\,,\dotsc, \,a_i)\in H^i(F\,,\,\mu_n^{\otimes i})$ will denote the cup product $(a_1)\cup \cdots\cup(a_i)\in H^i(F\,,\,\mu_n^{\otimes i})$. Such a
cohomology class is called a \empha{symbol class}.

By standard theories from Galois or \'etale cohomology, there are \empha{residue homomorphisms} for all $i\ge 1$ and all $j\in\mathbb{Z}$
\[
\partial_v^{i,\,j}\;:\;\; H^i(F\,,\,\mu_n^{\otimes j})\lra H^{i-1}(\kappa(v)\,,\,\mu_n^{\otimes (j-1)})
\]which fit into a long exact sequence
\[
\cdots\lra H^i_{\et}(\mathcal{O}_v\,,\,\mu_n^{\otimes j})\lra H^i(F\,,\,\mu_n^{\otimes j})\overset{\partial^{i,\,j}_v}{\lra} H^{i-1}(\kappa(v)\,,\,\mu_n^{\otimes (j-1)})\lra H^{i+1}_{\et}(\mathcal{O}_v\,,\,\mu_n^{\otimes j})\lra\cdots
\]
An element $\alpha\in H^i(F\,,\,\mu_n^{\otimes j})$ is called \empha{unramified} at $v$ if $\partial^{i,\,j}_v(\alpha)=0$.

\

Now consider the case of Brauer groups.  By Thm.$\;$\ref{thm2p1temp} (ii), $\Br(\mathcal{O}_v)$
gets identified with a subgroup of $\Br(F)$. An element $\alpha\in\Br(F)$ is called \empha{unramified} at
$v$ if it lies in the subgroup $\Br(\mathcal{O}_v)\subseteq\Br(F)$. If $n>0$ is a positive integer which is invertible in $\kappa(v)$, then an element
$\alpha\in \Br(F)[n]$ is unramified at $v$ if and only if $\partial_v(\alpha)=0$, where $\partial_v$ denotes the residue map
\[
\partial_v=\partial^{2,\,1}_v\;:\; \Br(F)[n]=H^2(F\,,\,\mu_n)\lra H^1(\kappa(v)\,,\,\mathbb{Z}/n)\,.
\]
As we will frequently speak of ramification of division algebras, the above residue map $\partial_v=\partial^{2,\,1}_v$ will often be called the
\empha{ramification map} and denoted by $\mathrm{ram}_v$.

Let $X$ be a scheme equipped with a morphism $\Spec F\to X$. The
subgroup
\[
\Br_{nr}(F/X):=\bigcap_{v\in\Omega(F/X)}\Br(\mathcal{O}_v)\;\subseteq\;\Br(F),
\]where $\Omega(F/X)$ denotes the set of discrete valuations of $K$ which have a center on $X$, is referred to as the (relative) \empha{unramified Brauer group} of $F$ over $X$. A Brauer class
$\alpha\in\Br(F)$ is called \empha{unramified over $X$} if it lies in the subgroup $\Br_{nr}(F/X)$. We say a field extension $M/F$
\empha{splits all ramification of $\alpha\in\Br(F)$ over $X$} if $\alpha_M\in\Br(M)$ is unramified over $X$. When $X=\Spec A$ is affine, we write
$\Br_{nr}(F/A)$ instead of $\Br_{nr}(F/\Spec A)$.

If $X$ is an integral scheme with function field $F$ and if $X\to Y$ is
 a proper morphism, then $\Omega(F/X)=\Omega(F/Y)$ and hence
 $\Br_{nr}(F/X)=\Br_{nr}(F/Y)$. If $X$ is a regular curve or surface
 with function field $F$, then Thm.$\;$\ref{thm2p1temp} implies that $\Br_{nr}(F/X)\subseteq \Br(X)$.

\

Note that for any field $\kappa$, the Galois cohomology group
$H^1(\kappa\,,\,\mathbb{Q}/\mathbb{Z})$ is identified with the group of characters
of the absolute Galois group $G_{\kappa}$, i.e., the group $\Hom_{cts}(G_{\kappa}\,,\,\mathbb{Q}/\mathbb{Z})$
of continuous homomorphisms $f: G_{\kappa}\to \mathbb{Q}/\mathbb{Z}$. Any character
$f\in \Hom_{cts}(G_{\kappa}\,,\,\mathbb{Q}/\mathbb{Z})$ must have image of the form
$\mathbb{Z}/m\subseteq\mathbb{Q}/\mathbb{Z}$ for some positive integer $m$ and its kernel is
equal to $G_{\kappa'}$, for some cyclic Galois extension
$\kappa'/\kappa$ of degree $m$. There is a generator
$\sigma\in\mathrm{Gal}(\kappa'/\kappa)$ such that $f(\sigma)=1+m\mathbb{Z}\in \mathbb{Z}/m$.
The function $f\in \Hom_{cts}(G_{\kappa}\,,\,\mathbb{Q}/\mathbb{Z})$ is uniquely
determined by the pair $(\kappa'/\kappa\,,\,\sigma)$. In this paper
we will often write an element of $H^1(\kappa\,,\,\mathbb{Q}/\mathbb{Z})$ in this
way. In particular, the ramification $\ram_v(\alpha)\in
H^1(\kappa(v)\,,\,\mathbb{Z}/n)$ of a Brauer class $\alpha\in\Br(F)[n]$ at a discrete valuation $v\in\Omega(F/X)$
will be represented in this way.

Let $\chi\in H^1(F\,,\,\mathbb{Q}/\mathbb{Z})=\Hom_{cts}(G_F\,,\,\mathbb{Q}/\mathbb{Z})$ be a character of $G_F$
with image $\mathbb{Z}/n\subseteq \mathbb{Q}/\mathbb{Z}$, represented by a pair
$(L/F\,,\,\sigma)$, i.e., $L/F$ is a finite cyclic Galois extension of degree $n$ such that
\[
G_L=\mathrm{Ker}(\chi\,:\;G_F\lra\mathbb{Q}/\mathbb{Z})\,
\]and $\sigma\in\mathrm{Gal}(L/F)$ is a generator such that $\chi(\sigma)=1+n\mathbb{Z}\in\mathbb{Z}/n$.
Recall that (\cite[$\S$2.5]{GS}) the \empha{cyclic algebra}
$(\chi\,,\,b)$ associated with $\chi$ and an element $b\in F^*$ is the $F$-algebra
generated by $L$ and a word $y$ subject to the following multiplication
relations
\[
y^n=b\,\quad\text{and }\;\;\quad \lambda y=y\sigma(\lambda),\;\;\forall\;\lambda\in L\,.
\]It is a standard fact that $(\chi\,,\,b)$ is a central simple algebra of degree $n$ over $F$. The class of the
cyclic algebra $(\chi\,,\,b)$ in $\Br(F)[n]=H^2(F\,,\,\mu_n)$
coincides with the cup product of $\chi\in H^1(F\,,\,\mathbb{Z}/n)$ and
$(b)\in H^1(F\,,\,\mu_n)$.

If $\mu_n\subseteq F$, then by Kummer theory $L$ is of the form
$L=F(\sqrt[n]{a})$ for some $a\in F^*$. There is a primitive $n$-th
root of unity $\xi_n\in F$ such that
$\sigma(\sqrt[n]{a})=\xi_n\sqrt[n]{a}$. The cyclic algebra
$(\chi\,,\,b)$ is isomorphic to the $F$-algebra $(a\,,\,b)_{\xi_n}$,
which by definition is the $F$-algebra generated by two words
$x\,,\,y$ subject to the relations
\[
x^n=a\,,\,\;y^n=b\,,\,\; xy=\xi_nyx\,.
\]
Conversely, when $F$ contains a primitive $n$-th root of unity $\xi_n$,
the algebra $(a\,,\,b)_{\xi_n}$ associated to elements $a\,,\,b\in F^*$ is
isomorphic to $(\chi\,,\,b)$, where $\chi\in H^1(F\,,\,\mathbb{Q}/\mathbb{Z})$ is the character
represented by the cyclic extension $L/F=F(\sqrt[n]{a})/F$ and the
$F$-automorphism $\sigma\in\mathrm{Gal}(L/F)$ which sends $\sqrt[n]{a}$ to $\xi_n\sqrt[n]{a}$.
The class of the algebra $(a\,,\,b)_{\xi_n}$
in $\Br(F)$ will be denoted by $(a\,,\,b)$ when the degree $n$ and the choice of $\xi_n\in F$
are clear from the context. This notation is compatible
with the notion of symbol classes via the isomorphism
$\Br(F)[n]=H^2(F\,,\,\mu_n)\cong H^2(F\,,\,\mu_n^{\otimes 2})$ corresponding to the
choice of $\xi_n\in F$.

\section{Division algebras over local henselian surfaces}\label{sec3}

In this section we first recall a number of techniques in Saltman's method
of detecting ramification of division algebras (\cite{Salt97},
\cite{Salt07}) and then we will prove Theorem$\;$\ref{thm1p3temp}.

\subsection{Ramification of division algebras over surfaces}\label{sec3p1}

In this subsection, let $X$ be a regular excellent surface and let
$F$ be the function field of $X$. By resolution of embedded
singularities (\cite[p.38, Thm. and p.43, Remark]{Sha66},
\cite[p.193]{Lip75}), for any effective divisor $D$ on $X$, there
exists a regular surface $X'$ together with a proper birational
morphism $X'\to X$, obtained by a sequence of blow-ups, such that
the total transform $D'$ of $D$ in $X'$ is a \empha{simple normal
crossing} (snc) divisor (i.e., the reduced subschemes on the
irreducible components of $D'$ are regular curves and they meet
transversally everywhere). We will use this result without further
reference.

Let $n$ be a positive integer that is invertible on $X$ and let $\alpha\in \Br(F)[n]$
be a Brauer class of order dividing $n$. For any discrete valuation $v\in\Omega(F/X)$, let $\ram_v$ denote the ramification map (or the residue map)
\[
\ram_v=\partial^{2,\,1}_v\;:\; \Br(F)[n]=H^2(F\,,\,\mu_n)\lra H^1(\kappa(v)\,,\,\mathbb{Z}/n)\,.
\]If $v=v_C$ is the discrete valuation centered at the
generic point of a curve $C\subseteq X$, we write
$\ram_C=\ram_{v_C}$. The \empha{ramification locus} of
$\alpha\in\Br(F)[n]$ on $X$, denoted $\Ram_X(\alpha)$, is by definition
the (finite) union of curves $C\subseteq X$ such that
$\ram_C(\alpha)\neq 0\in H^1(\kappa(C)\,,\,\mathbb{Z}/n)$. The
\empha{ramification divisor} of $\alpha$ on $X$, denoted again by
$\Ram_X(\alpha)$ by abuse of notation, is the reduced divisor supported
on the ramification locus. After several blow-ups, we may assume
$\Ram_X(\alpha)$ is an snc divisor on $X$.

\begin{defn}[{\cite[$\S$2]{Salt07}}]\label{defn3p1temp}
Let $X\,,\,F$ and $\alpha$ be as above. Assume that $\Ram_X(\alpha)$ is an snc divisor on $X$.
A closed point $P\in X$ is called

(1) a \empha{distant point} for $\alpha$, if $P\notin\Ram_{X}(\alpha)$;

(2) a \empha{curve point} for $\alpha$, if $P$ lies on one and only one irreducible component of $\Ram_X(\alpha)$;

(3) a \empha{nodal point} for $\alpha$, if $P$ lies on two different irreducible components of $\Ram_X(\alpha)$.
\end{defn}

From a local study of a Brauer class at closed points in its ramification locus (\cite[Prop.$\;$1.2]{Salt97}), Saltman essentially proved the following theorem.

\begin{thm}[{\cite[Thm.$\;$2.1]{Salt97}, \cite{Salt98}}]\label{thm3p3temp}
Let $X$ be a regular excellent surface which is quasi-projective
over a ring, $F$ the function field of $X$, $n>0$ a positive integer
that is invertible on $X$, and $\alpha\in \Br(F)[n]$. Assume
$\mu_n\subseteq F$.

Then there exist $f\,,\,g\in F^*$ such that the field extension
$M/F:=F(\sqrt[n]{f}\,,\,\sqrt[n]{g})/F$ splits all ramification of
$\alpha$ over $X$, i.e., $\alpha_M\in\Br_{nr}(M/X)$.
\end{thm}

Although the setup of the above theorem differs from that of
\cite[Thm.$\;$2.1]{Salt97}, a careful verification shows that
Saltman's (corrected) proof (with corrections pointed out by Gabber, cf. \cite{Salt98}) still works in our
context. One can also find a proof in Brussel's notes
\cite[Lemma$\;$7.8]{Bru}. In the case where $n$ is prime, a stronger statement holds. (See \cite[Thm.$\;$7.13]{Salt08} and Prop.$\;$\ref{prop3p14temp}.)

\begin{remark}
Let $R$ be a $2$-dimensional henselian, excellent local domain with fraction
field $K$ and residue field $k$.  By resolution of
singularities for surfaces (see \cite{Lip75}, \cite{Lip78}), there exists a regular
proper model $\mathcal{X}\to\Spec R$. The structural morphism
$\mathcal{X}\to\Spec R$ is actually projective by \cite[IV.21.9.13]{EGA4}. So
Thm.$\;$\ref{thm3p3temp} applies to such a regular proper model $\mathcal{X}\to \Spec R$.

If the residue field $k$ of $R$ is finite,  Thm.$\;$\ref{thm3p3temp} has  the following
refined form over the fraction field $K$:

Let $n>0$ be a positive integer that is invertible in the finite residue field $k$. Assume that $\mu_n\subseteq R$. Then for any finite collection of Brauer classes $\alpha_i\in \Br(K)[n]\,,\,1\le i\le m$,
there exist $f,\,g\in K^*$ such that the field extension $M/K:=K(\sqrt[n]{f}\,,\,\sqrt[n]{g})/K$
splits all the $\alpha_i\,,\,i=1\,,\dotsc, m$.

In the literature, this result has been established in the case where $K$ is a function field of a $p$-adic curve and where $n$ is a prime number, and the proof is essentially an observation of Gabber and Colliot-Th\'el\`ene  (cf. \cite{CT98}, \cite[Thm.$\;$2.5]{HVG}). One may verify that essentially the same arguments work in the local henselian case considered here.
\end{remark}

We will need the following analog of \cite[Thm.$\;$3.4]{Salt97} in the sequel.

\begin{thm}\label{thm3p6temp}
Let $R$ be a $2$-dimensional, henselian, excellent local domain with
fraction field $K$ and residue field $k$. Let $n>0$ be a positive
integer that is invertible in $k$. Assume that $k$ is a $B_1$ field.

Then any Brauer class $\alpha\in\Br(K)$ of order $n$ has index dividing $n^2$.
\end{thm}
\begin{proof}
This follows on parallel lines along the proof of \cite[Thm.$\;$3.4]{Salt97}
 with suitable substitutions of the ingredients used in the case of $p$-adic function fields.
For the sake of the reader's convenience, we recall the argument.

We may assume $n=q^r$ is a power of a prime number $q$. Let $\mathcal{X}\to\Spec R$ be
a regular proper model. For any finite field extension $K'/K$, the integral closure $R'$ of $R$ in $K'$
satisfies the same assumptions as $R$ and $K'$ is the function field of a regular proper model $\mathcal{X}'$ of $\Spec R'$.
One has $\Omega(K'/\mathcal{X}')=\Omega(K'/\mathcal{X})$ and
 $\Br_{nr}(K'/\mathcal{X})=\Br_{nr}(K'/\mathcal{X}')=0$ by Thm.$\;$\ref{thm2p1temp} (iii) and Coro.$\;$\ref{coro2p5temp}. So it suffices to
 find a finite separable field extension $K'/K$ of degree $q^{2r}m$ with $q\nmid m$
 such that $K'/K$ splits all ramification of $\alpha$ over $\mathcal{X}$.

Now we proceed by induction on $r$. First assume $r=1$. Then the result  is immediate from Thm.$\;$\ref{thm3p3temp} if $\mu_q\subseteq F$. The general case
follows by passing to the extension $F(\mu_q)/F$, which has degree prime to $q$.

For general $r$, the inductive
hypothesis applied to the Brauer class $q \alpha$ implies that there is
a separable field extension $K'/K$ splitting all ramification of
$q\alpha$ over $\mathcal{X}$, which has degree $q^{2r-2}m'$, where $q\nmid m'$.
But $q\alpha_{K'}=0\in\Br(K')$ by Coro.$\;$\ref{coro2p5temp}. By the case with $r=1$, we can find a separable extension
$K''/K'$ of degree $q^2m''$ with $q\nmid m''$ that splits all
ramification of $\alpha_{K'}$ over $\mathcal{X}$. Now $K''/K$ is a separable
extension of degree $[K'':K]=q^{2r}m$ with $m=m'm''$ coprime to $q$
and $K''/K$ splits all ramification of $\alpha$ over $\mathcal{X}$, as desired.
\end{proof}

We will give an example of a Brauer class $\alpha\in\Br(K)$ of order
$n$ which is of index $n^2$ in Example$\;$\ref{exam3p21}.

\subsection{Classification of nodal points}\label{sec3p2}

To prove further results, we need more analysis on ramification at nodal points, for which we briefly recall in this
subsection some basic notions and results due to Saltman. The reader is referred to \cite[$\S\S$2--3]{Salt07} or \cite[$\S\S$7--8]{Bru} for more
details.

Let $X$ be a regular excellent surface with function field $F$ and
let $q$ be a prime number which is invertible on $X$. Let
$\alpha\in\Br(F)[q]$. Assume that $\Ram_X(\alpha)$ is an snc divisor on
$X$. Let $P\in X$ be a nodal point for $\alpha$ (cf.
Definition$\;$\ref{defn3p1temp}), lying on two distinct irreducible
components $C_1,\,C_2$ of $\Ram_X(\alpha)$. Let
$\chi_1=\ram_{C_1}(\alpha)$ and $\chi_2=\ram_{C_2}(\alpha)$ be
respectively the ramification of $\alpha$ at $C_1,\,C_2$. Since the
natural sequence induced by residue maps
\[
H^2(F\,,\,\mu_q)\lra \bigoplus_{v\in(\Spec \mathscr{O}_{X,\,P})^{(1)}}H^1(\kappa(v)\,,\,\mathbb{Z}/q)\lra
 H^0(\kappa(P)\,,\,\mu_q^{\otimes (-1)})
\]
is a complex (cf. \cite{Ka86} or \cite[Prop.$\;$2.3]{CT06}), $\chi_1=\ram_{C_1}(\alpha)\in H^1(\kappa(C_1)\,,\,\mathbb{Z}/q)$
is unramified at $P$ if and only if $\chi_2=\ram_{C_2}\in H^1(\kappa(C_2)\,,\,\mathbb{Z}/q)$ is unramified at $P$.

\begin{defn}[{\cite[$\S\S$2--3]{Salt07}}]\label{defn3p7temp}
Let $X\,,\,F\,,\,q\,,\,\alpha$ and so on be as above. Assume that $\Ram_X(\alpha)$ is an snc divisor on $X$. Let $P\in X$ be a nodal point for
$\alpha$, lying on two distinct irreducible components $C_1,\,C_2$ of $\Ram_X(\alpha)$.

(1) $P$ is called a \empha{cold point} for $\alpha$ if $\chi_1=\ram_{C_1}(\alpha)\in H^1(\kappa(C_1)\,,\,\mathbb{Z}/q)$ (and hence also
$\chi_2=\ram_{C_2}(\alpha)\in H^1(\kappa(C_2)\,,\,\mathbb{Z}/q)$) is ramified at $P$.

(2) Assume now $\chi_1$ and $\chi_2$ are unramified at $P$, so that
they lie respectively in $H^1(\mathscr{O}_{C_1\,,\,P}\,,\,\mathbb{Z}/q)$ and
$H^1(\mathscr{O}_{C_2\,,\,P}\;,\,\mathbb{Z}/q)$. Let $\chi_i(P)\in
H^1(\kappa(P)\,,\,\mathbb{Z}/q)\,,\,i=1,\,2$ be their specializations and
let $\langle\chi_i(P)\rangle$, $i=1,\,2$ be the subgroups of
$H^1(\kappa(P)\,,\,\mathbb{Z}/q)$ generated by $\chi_i(P)$ respectively.
Then $P$ is called

(2.a) a \empha{cool point} for $\alpha$ if $\langle\chi_1(P)\rangle=\langle\chi_2(P)\rangle=0$;

(2.b) a \empha{chilly point} for $\alpha$ if $\langle\chi_1(P)\rangle=\langle\chi_2(P)\rangle\neq 0$;

(2.c) a \empha{hot point} for $\alpha$ if $\langle\chi_1(P)\rangle\neq \langle\chi_2(P)\rangle$.

When $P$ is a chilly point, there is a unique $s=s(C_2/C_1)\in
(\mathbb{Z}/q)^*$ such that
\[
\chi_2(P)=s.\chi_1(P)\in H^1(\kappa(P)\,,\,\mathbb{Z}/q)\,.
\]One says that
$s=s(C_2/C_1)$ is the \empha{coefficient} of the chilly point $P$
with respect to $C_1$.
\end{defn}

\begin{remark}\label{remark3p8temp}
One may verify without much pain that our classification of nodal
points, following \cite[Definition$\;$8.5]{Bru}, is equivalent to
Saltman's original formulation, which goes as follows: First
consider the case $\mu_q\subseteq F$. Then
\[
\alpha\equiv(u\,,\,\pi)+(v\,,\,\delta)+r.(\pi\,,\,\delta)\;\pmod{\Br(\mathscr{O}_{X,\,P})}
\]by  \cite[Prop.$\;$1.2]{Salt97}. Here $u,\,v\in\mathscr{O}_{X,\,P}^*$, $r\in\mathbb{Z}/q$
and $\pi,\,\delta\in\mathscr{O}_{X,\,P}$ are local
equations of the two components of $\Ram_X(\alpha)$ passing through
$P$. The point $P$ is a \empha{cold point} if $r\neq 0\in\mathbb{Z}/q$.
Assume next $r=0\in\mathbb{Z}/q$. Then $P$ is a \empha{cool point} if
$u(P)\,,\,v(P)\in \kappa(P)^{*q}$; a \empha{chilly point} if
$u(P)\,,\,v(P)\notin \kappa(P)^{*q}$ and they generate the same
subgroup of $\kappa(P)^*/\kappa(P)^{*q}$; or a \empha{hot point}
otherwise. In the general case, let $X'\to X$ be the connected
finite \'etale cover obtained by adjoining all $q$-th roots of unity
and let $\alpha'$ be the canonical image of $\alpha$ in $\Br(F')$, where
$F'$ denotes the function field of $X'$. Then for any two points
$P_1'\,,\,P_2'\in X'$, both lying over $P\in X$, $P_1'$ is a cold
(resp. cool, resp. chilly, resp. hot) point for $\alpha'$ if and only
if $P'_2$ is, and in that case one says that $P$ is a cold (resp.
cool, resp. chilly, resp. hot) point for $\alpha$. When $P$ is chilly,
the coefficient of $P$ with respect to a component through it is
also well-defined, as the coefficient of any preimage $P'$ of $P$.
\end{remark}

To get some compatibility for coefficients of chilly points, one has
to eliminate the so-called \empha{chilly loops}, i.e., loops in the
following graph: The set of vertices is the set of irreducible
components of $\Ram_X(\alpha)$ and the number $r\ge 0$ of edges linking two vertices is equal to the number of chilly points
in the intersection of the two curves corresponding to the two vertices. (Two vertices may be joined by two or more edges and thus contribute to some loops.)

\begin{prop}[{\cite[Prop.$\;$3.8]{Salt07}}]\label{prop3p9temp}
Let $X\,,\,F\,,\,q$ and $\alpha\in \Br(F)[q]$ be as above. Assume that
$\Ram_X(\alpha)$ is an snc divisor on $X$. Then there exists a proper
birational morphism $X'\to X$, obtained by a finite number of
blow-ups, such that $\alpha$ has no cool points and no chilly loops on
$X'$.
\end{prop}

We also need the notion of \empha{residual class} at a ramified
place. Let $C$ be an irreducible component of $\Ram_X(\alpha)$ and let
$v=v_C$ be the associated discrete valuation of $F$. Choose any
$x\in F^*$ with $q\nmid v(x)$, so that the extension
$M/F:=F(\sqrt[q]{x})/F$ is totally ramified at $v=v_C$ and
$\alpha_M=\alpha\otimes_FM\in \Br(M)$ is unramified at the unique discrete
valuation $w$ of $M$ that lies over $v$. One has
$\kappa(w)=\kappa(v)=\kappa(C)$ and hence a well-defined Brauer
class $\beta_{C,\,x}\in \Br(\kappa(C))$ given by the specialization
of $\alpha_M\in\Br(M)$ in $\Br(\kappa(w))=\Br(\kappa(C))$. Let
$(L/\kappa(C)\,,\,\sigma)=\ram_C(\alpha)$ be the ramification of $\alpha$
at $C$. Whether or not $\beta_{C,\,x}\in\Br(\kappa(C))$ is split by
the field extension $L/\kappa(C)$ does not depend on the choice of
$M=F(\sqrt[q]{x})$ (\cite[Coro.$\;$0.7]{Salt07}). We say that the
\empha{residual classes of $\alpha$ at $C$ are split by the
ramification}, if for one (and hence for all) choice of
$M=F(\sqrt[q]{x})$, the residual class
$\beta_{C,\,x}\in\Br(\kappa(C))$ is split by $L/\kappa(C)$
(\cite[p.821, Remark]{Salt07}). When we are only interested in this
property, we will simply write $\beta_C$ for
$\beta_{C,\,x}\in\Br(\kappa(C))$ with respect to any choice of $x$.

It is proved in \cite[Prop.$\;$0.5 and Prop.$\;$3.10 (d)]{Salt07} that if $\alpha$ has index $q$, then all the residual classes $\beta_C$ of $\alpha$ at all components $C$ of $\Ram_X(\alpha)$ are split by the ramification and there are no hot points for $\alpha$ on $X$.

\subsection{Splitting over a Kummer extension}\label{sec3p3}

Let $X$ be a reduced scheme which is projective over a ring. Let $\mathcal{P}\subseteq X$ be a finite set of closed points of $X$. Denote by
$\mathscr{K}_X$ the sheaf of meromorphic functions on $X$ and set $\mathcal{P}^*:=\oplus_{P\in\mathcal{P}}\kappa(P)^*$. Let $\mathscr{O}_{X,\,\mathcal{P}}^*$ denote the kernel of the natural
surjection of sheaves $\mathscr{O}_X^*\to\mathcal{P}^*$, so that there is a natural exact sequence
\[
1\lra \mathscr{O}_{X\,,\,\mathcal{P}}^*\lra\mathscr{O}_X^*\lra \mathcal{P}^*\lra 1\,.
\]Define subgroups $K_{\mathcal{P}}^*\subseteq H^0(X\,,\,\mathscr{K}^*_X)$ and $H_{\mathcal{P}}^0(X\,,\,\mathscr{K}^*_X/\mathscr{O}_X^*)\subseteq H^0(X\,,\,\mathscr{K}^*_X/\mathscr{O}_X^*)$ by
\[
\begin{split}
&K_{\mathcal{P}}^*:=\set{f\in H^0(X\,,\,\mathscr{K}^*_X)\,|\, f\in\mathscr{O}_{X\,,\,P}^*\;,\;\;\forall\; P\in\mathcal{P}}\\
\text{and}\quad & H^0_{\mathcal{P}}(X\,,\,\mathscr{K}^*_X/\mathscr{O}_X^*):=\set{D\in H^0(X\,,\,\mathscr{K}^*_X/\mathscr{O}_X^*)\,|\,\mathrm{Supp}(D)\cap\mathcal{P}=\emptyset}\,.
\end{split}
\]Consider the natural map
\[
\phi\;:\;\; K_{\mathcal{P}}^*\lra H^0_{\mathcal{P}}(X\,,\,\mathscr{K}^*_X/\mathscr{O}_X^*)\oplus \left(\oplus_{P\in\mathcal{P}}\kappa(P)^*\right)\;;\quad f\longmapsto (\di_X(f)\,,\,\oplus f(P))\,.
\]

\begin{prop}[{\cite[Prop.$\;$1.6]{Salt07}}]\label{prop3p11temp}
With notation as above, there is a natural isomorphism
\[
H^1(X\,,\,\mathscr{O}_{X,\,\mathcal{P}}^*)\cong \frac{H^0_{\mathcal{P}}(X\,,\,\mathscr{K}^*_X/\mathscr{O}_X^*)\oplus \left(\oplus_{P\in\mathcal{P}}\kappa(P)^*\right)}{\phi(K_{\mathcal{P}}^*)}
\]
\end{prop}

The analog in the arithmetic case of the following proposition is  \cite[Prop.$\;$1.7]{Salt07}.
The following generalization to the case where $A$ is $2$-dimensional will be indispensable in the proofs of our results.

\begin{prop}\label{prop3p12temp}
Let $A$ be a $($noetherian$)$ normal, henselian local domain with residue
field $\kappa$, $X$ an integral scheme, and $X\to\Spec A$ a
proper morphism  whose closed fiber $X_0$ has dimension
$\le 1$ and whose generic fiber is geometrically integral. Let
$m$ be a positive integer that is invertible in $A$. Let
$\ov{X}=(X_0)_{\red}$ be the reduced closed subscheme on the
closed fiber $X_0$. Suppose that $\ov{X}$ is geometrically
reduced over $\kappa\,($e.g. $\kappa$ is perfect$)$.

Then for any finite set $\mathcal{P}$ of closed points of $X$, the natural
map
\[
H^1(X\,,\,\mathscr{O}_{X\,,\,\mathcal{P}}^*)\lra
H^1(\ov{X}\,,\,\mathscr{O}_{\ov{X}\,,\,\mathcal{P}}^*)
\]is surjective and induces a canonical isomorphism
\[
H^1(X\,,\,\mathscr{O}_{X\,,\,\mathcal{P}}^*)/m\cong
H^1(\ov{X}\,,\,\mathscr{O}^*_{\ov{X}\,,\,\mathcal{P}})/m\,.
\]
\end{prop}

To prove the above proposition, we need a well-known lemma.

\begin{lemma}\label{lemma3p13temp}
Let $A$ be a $($noetherian$)$ henselian local ring, $X\to\Spec A$ a
proper morphism with closed fiber $X_0$ of dimension $\le 1$,
$m>0$ a positive integer that is invertible in $A$ and
$\ov{X}=(X_0)_{\red}$ the reduced closed subscheme on the closed
fiber $X_0$.

Then the natural map $\Pic(X)\to\Pic(\,\ov{X}\,)$ is surjective
and induces an isomorphism
\[
\Pic(X)/m\simto\Pic(\,\ov{X}\,)/m\,.
\]
\end{lemma}
\begin{proof}
The surjectivity of $\Pic(X)\to\Pic(\,\ov{X}\,)$ follows from
\cite[IV.21.9.12]{EGA4}. Then the commutative diagram with exact rows, which comes from the Kummer sequence,
\[
\begin{CD}
0 @>>> \Pic(X)/m @>>> H^2_{\et}(X\,,\,\mu_m) @>>> \Br(X)[m]
@>>>
0\\
&& @VVV  @VVV  @VVV \\
0 @>>> \Pic(\,\ov{X}\,)/m @>>> H^2_{\et}(\,\ov{X}\,,\,\mu_m)@>>>
\Br(\ov{X})[m]@>>> 0
\end{CD}
\]yields the desired isomorphism
$\Pic(X)/m\simto\Pic(\,\ov{X}\,)/m$, since the vertical map in
the middle is an isomorphism by proper base change (cf. \cite[p.224, Coro. 2.7]{Mil80}), noticing also that any scheme $Y$
has the same \'etale cohomology with value in a commutative \'etale group scheme as
its reduced closed subscheme $Y_{\red}$ (cf. \cite[Exp.$\;$VIII, Coro.$\;$1.2]{SGA}).
\end{proof}

\begin{proof}[Proof of Prop.$\;\ref{prop3p12temp}$]
Consider the commutative diagram with exact rows
\[
\begin{CD}
H^0(X\,,\,\mathscr{O}^*) @>{\varphi}>> H^0(X\,,\,\mathcal{P}^*) @>>>
H^1(X\,,\,\mathscr{O}_{X\,,\,\mathcal{P}}^*) @>>> \Pic(X) @>>> 0\\
@V{\pi}VV @VV{\mathrm{id}}V @VVV @VVV \\
H^0(\ov{X}\,,\,\mathscr{O}^*) @>{\theta}>> H^0(\ov{X}\,,\,\mathcal{P}^*) @>>>
H^1(\ov{X}\,,\,\mathscr{O}_{\ov{X}\,,\,\mathcal{P}}^*) @>>> \Pic(\,\ov{X}\,)
@>>> 0
\end{CD}
\]from which the surjectivity of $H^1(X\,,\,\mathscr{O}_{X\,,\,\mathcal{P}}^*)\to H^1(\ov{X}\,,\,\mathscr{O}_{\ov{X}\,,\,\mathcal{P}}^*)$
is immediate, since $\Pic(X)\to\Pic(\ov{X})$ is surjective by Lemma$\;$\ref{lemma3p13temp}. Put $M:=\mathrm{Im}(\varphi)\subseteq N:=\mathrm{Im}(\theta)$.

We claim that $\pi$ is surjective. Indeed, by Zariski's connectedness
theorem (\cite[III.4.3.12]{EGA3}), the hypotheses that $A$ is normal
and the generic fiber of $X\to\Spec A$ is geometrically integral
imply that the closed fiber $X_0$ is geometrically connected.
The reduced closed fiber $\ov{X}=(X_0)_{\red}$ is geometrically
connected as well. Since $\ov{X}$ is assumed to be geometrically
reduced, we have $H^0(\,\ov{X}\,,\,\mathscr{O}^*)=\kappa^*$. Thus, the map
$\pi: H^0(X\,,\,\mathscr{O}^*)\to H^0(\,\ov{X}\,,\,\mathscr{O}^*)$ is clearly
surjective since $A^*\subseteq H^0(X\,,\,\mathscr{O}^*)$.

Now our claim shows that $M=N$ and then it follows that
\[
\mathrm{Ker}\left(H^1(X\,,\,\mathscr{O}_{X\,,\,\mathcal{P}}^*)\to
H^1(\,\ov{X}\,,\,\mathscr{O}_{\ov{X}\,,\,\mathcal{P}}^*)\right)\cong
B:=\mathrm{Ker}(\Pic(X)\to\Pic(\,\ov{X}\,))
\]It's sufficient to show $B/m=0$. From the commutative diagram with
exact rows
\[
\begin{CD}
0 @>>> H^0(X\,,\,\mathscr{O}^*)/m @>>> H^1_{\et}(X\,,\,\mu_m) @>>>
\Pic(X)[m] @>>> 0\\
&& @VV{\pi}V @VVV @VVV \\
0 @>>> H^0(\,\ov{X}\,,\,\mathscr{O}^*)/m @>>>
H^1_{\et}(\,\ov{X}\,,\,\mu_m) @>>> \Pic(\,\ov{X}\,)[m] @>>> 0
\end{CD}
\]
it follows that $\Pic(X)[m]\cong\Pic(\,\ov{X}\,)[m]$, since the
vertical map in the middle is an isomorphism by proper base change
and the left vertical map is already shown to be surjective. Now
applying the snake lemma to the following commutative diagram
\[
\begin{CD}
0 @>>> B @>>> \Pic(X)@>>>
\Pic(\ov{X}) @>>> 0\\
&& @VV{m}V @VV{m}V @VV{m}V \\
0 @>>> B@>>> \Pic(X)@>>> \Pic(\ov{X}) @>>> 0
\end{CD}
\]
 and
using Lemma$\;$\ref{lemma3p13temp}, we easily find $B/m=0$, which
completes the proof.
\end{proof}

The following result is proved in \cite[Thm.$\;$7.13]{Salt08} in the case where $\mu_q\subseteq F$ without assuming the residue field $\kappa$
perfect. It says essentially that the conclusion of Theorem$\;$\ref{thm3p3temp} can be strengthened for Brauer classes of prime order.

\begin{prop}\label{prop3p14temp}
Let $A$ be a $($noetherian$)$ henselian local domain with
residue field $\kappa$, $q$ a prime number unequal to the
characteristic of $\kappa$, and $X$ a regular excellent surface equipped with a proper dominant morphism $X\to\Spec A$ whose closed fiber is of
dimension $\le 1$. Let $F$ be the
function field of $X$ and $\alpha\in\Br(F)[q]$. Assume that $\kappa$ is perfect and that $\alpha$ has index $q$.

Then there is some $g\in F^*$ such that the field extension
$M/F:=F(\sqrt[q]{g})/F$ splits all ramification of $\alpha$ over $X$,
i.e., $\alpha_M\in\Br_{nr}(M/X)$.
\end{prop}
\begin{proof}Replacing $A$ by its normalization if necessary, we may assume that $A$ is normal.
Let $\mathrm{Ram}_{X}(\alpha)=\sum C_i$ be the ramification divisor of
$\alpha$ on $X$ and let $\ov{X}=(X_0)_{\red}$ be the reduced closed
subscheme on the closed fiber $X_0$. After a finite number of
blow-ups, we may assume that $X_0$ is purely of dimension 1, that
$B:=(\mathrm{Ram}_{X}(\alpha)\cup\ov{X})_{\red}$ is an snc divisor, and
that there are no cool points or chilly loops for $\alpha$ on $X$ (cf.
Prop.$\;$\ref{prop3p9temp}).  Write
\[
(L_i/\kappa(C_i)\,,\,\sigma_i)=\ram_{C_i}(\alpha)\in H^1(\kappa(C_i)\,,\,\mathbb{Z}/q)
\]for the ramification of $\alpha$ at $C_i$. By the
assumption on the index, there are no hot points for $\alpha$ on $X$
and the residual classes of $\alpha$ at $C_i$ are split by the ramification $L_i/\kappa(C_i)$ for every
$i$ (cf. \cite[Prop.$\;$0.5 and Prop.$\;$3.10 (d)]{Salt07}). Using \cite[Thm.$\;$4.6]{Salt07}, we can find $\pi\in F^*$ having
the following properties:

(P.1) $v_{C_i}(\pi)=s_i$ is not divisible by $q$.

(P.2) If $P$ is a chilly point in the intersection of $C_i$ and
$C_j$, then the coefficient $s(C_j/C_i)$ of $P$ with
respect to $C_i$ (cf. Definition$\;$\ref{defn3p7temp}) satisfies $s(C_j/C_i)s_i=s_j\in\mathbb{Z}/q\mathbb{Z}$.

(P.3) The divisor $E:=\di_{X}(\pi)-\sum s_iC_i$ does not contain
any singular points of $B=(\mathrm{Ram}_{X}(\alpha)\cup
X_0)_{\red}$ or any irreducible component of $B$ in its support.

(P.4) With respect to $F':=F(\pi^{1/q})$, the residue Brauer
classes $\beta_{C_i\,,\,F'}=\beta_{C_i\,,\,\pi}\in\Br(\kappa(C_i))$ of $\alpha$ at all the $C_i$ are trivial.

(P.5) For any closed point $P$ in the intersection of $E$ and some
$C_i$, the intersection multiplicity $(C_i\cdot E)_P$ is a multiple
of $q$ if the corresponding field extension $L_i/\kappa(C_i)$ is
nonsplit at $P$.

Let $\gamma\in\Pic(X)$ be the class of $\mathscr{O}_{X}(-E)$ and let
$\ov{\gamma}\in\Pic(\ov{X})$ be its canonical image. By property (P.3), $E$
and $\ov{X}$ only intersect in nonsingular points of $\ov{X}$.
So we can represent $\ov{\gamma}$ as a Cartier divisor on $\ov{X}$
using the intersection of $-E$ and $\ov{X}$. This divisor can be
chosen in the following form
\begin{equation}\label{eq1temp}
\sum qn_jQ_j+\sum n_l Q'_l\,,
\end{equation}
where $Q_j\,,\,Q'_l$ are nonsingular points on $\ov{X}$, and for
each $Q'_l$, one has either $Q'_l\notin\mathrm{Ram}_{X}(\alpha)$ or
$Q'_l\in C_i$ for exactly one $C_i$ and the corresponding field
extension $L_i/\kappa(C_i)$ is split at $Q'_l$ (by property (P.5)).

By \cite[IV.21.9.11 and IV.21.9.12]{EGA4}, there exists a prime
divisor $E'_l$ on $X$ such that $E'_l|_{\ov{X}}=Q'_l$ as Cartier
divisors on $\ov{X}$. Note that $E'_l\nsubseteq
\mathrm{Ram}_{X}(\alpha)$ because otherwise $Q'_l\in E'_l\cap
\ov{X}$ would be a singular point of $E'_l\cup \ov{X}\subseteq
B=(\mathrm{Ram}_{X}(\alpha)\cup\ov{X})_{\red}$. Set $E'=-E-\sum
n_lE'_l$. Let $\mathcal{P}$ be the set of all singular points of $B$ (in
particular, $\mathcal{P}$ contains all nodal points for $\alpha$).

Let $\gamma'\in H^1(X\,,\,\mathscr{O}_{X\,,\,\mathcal{P}}^*)$ be the element
represented by the pair
\[
(E'\,,\,\oplus 1)\in H^0_{\mathcal{P}}(X\,,\,\mathscr{K}^*/\mathscr{O}^*)\oplus
\left(\oplus_{P\in\mathcal{P}}\kappa(P)^*\right)
\] via the isomorphism
\[
H^1(X\,,\,\mathscr{O}_{X\,,\,\mathcal{P}}^*)\cong
\frac{H^0_{\mathcal{P}}(X\,,\,\mathscr{K}^*/\mathscr{O}^*)\oplus\left(\oplus_{P\in\mathcal{P}}\kappa(P)^*\right)}{K^*_{\mathcal{P}}}
\]in Prop.$\;$\ref{prop3p11temp}. (Here $E'\in H^0_{\mathcal{P}}(X\,,\,\mathscr{K}^*/\mathscr{O}^*)$ since by the choice of $\pi$,
$E$ does not contain any singular points of
$B=(\mathrm{Ram}_{X}(\alpha)\cup \ov{X})_{\red}$.) The image
$\ov{\gamma}'$ of $\gamma'$ in
$H^1(\ov{X}\,,\,\mathscr{O}^*_{\ov{X}\,,\,\mathcal{P}})$ lies in $q.
H^1(\ov{X}\,,\,\mathscr{O}^*_{\ov{X}\,,\,\mathcal{P}})$ by the expression
\eqref{eq1temp}.

From Prop.$\;$\ref{prop3p12temp} it follows that $\gamma'\in q.
H^1(X\,,\,\mathscr{O}_{X\,,\,\mathcal{P}}^*)$. Thus, by Prop.$\;$\ref{prop3p11temp},
there is a divisor $E''\in H^0_{\mathcal{P}}(X\,,\,\mathscr{K}^*/\mathscr{O}^*)$, elements
$a(P)\in \kappa(P)^*$ for each $P\in\mathcal{P}$, and $f\in F^*$ such that
$f$ is a unit at every $P\in \mathcal{P}$, $\di_{X}(f)=E'+qE''$ and
$f(P)=a(P)^q\,,\,\forall\; P\in\mathcal{P}$. We now compute
\begin{equation}\label{eq2temp}
\begin{split}
\di_{X}(f\pi)&=\di_{X}(f)+\di_{X}(\pi)=(E'+qE'')+\left(\sum
s_iC_i+E\right)\\
&=-E-\sum n_lE'_l+qE''+E+\sum s_iC_i\\
&=\sum s_iC_i+\left(qE''-\sum n_lE'_l\right)\\
&=:\sum s_iC_i+\sum \tilde{n}_j D_j\,.
\end{split}
\end{equation}Note that for any $D_j$, the following properties hold:

(P.6) $D_j$ can only intersect $B$ in nonsingular points of $B$.

(P.7) If $q\nmid \tilde{n}_j$, then $D_j\in \set{E'_l}$, so that
either $D_j\cap\mathrm{Ram}_{X}(\alpha)=\emptyset$ or $D_j\cap
\mathrm{Ram}_{X}(\alpha)$ consists of a single point $P$ which lies
on one $C_i$ and the corresponding field extension $L_i/\kappa(C_i)$
splits at $P$.

Now we claim that $g=f\pi$ satisfies the required property. That is,
putting $M=F((f\pi)^{1/q})$, $\alpha_M\in\Br(M)$ is unramified at every
discrete valuation of $M$ that lies over a point or a curve on
$X$.

Consider a discrete valuation of $M$ lying over some
$v\in\Omega(F/X)$.

If $v$ is centered at some $C_i$, then $M/F$ is totally ramified at
$v$ since the coefficient $s_i$ of $f\pi$ at $C_i$ is prime to $q$
(cf. \eqref{eq2temp}), hence in particular $M/F$ splits the
ramification of $\alpha$ at $v$. As $\alpha$ is unramified at all other
curves on $X$, we may restrict to the case where $v$ is centered at
a closed point $P$ of $X$. By \cite[Thm.$\;$3.4]{Salt07}, we can
also ignore distant points and curve points $P\in C_i$ where
$L_i/\kappa(C_i)$ splits at $P$.

Now assume that $P$ is a curve point lying on some $C_1\in\set{C_i}$
where the corresponding field extension $L_1/\kappa(C_1)$ is
nonsplit at $P$. By property (P.7), the only curves other than $C_1$
in the support of $\di_{X}(f\pi)$ that can pass through $P$ have
coefficients a multiple of $q$. Therefore, in $R_P=\mathscr{O}_{X\,,\,P}$ we
have $f\pi=u\pi_1^{s_1}\delta^q$ with $u\in R_P^*$, $\pi_1\in R_P$ a
uniformizer of $C_1$ at $P$ and $\delta\in R_P$ prime to $\pi_1$.
Using \cite[Prop.$\;$3.5]{Salt07}, we then conclude that $M/F$
splits all ramification of $\alpha$ at $v$.

Recall that we have assumed there are no cool points or hot points
for $\alpha$. So in the only remaining cases, $P$ is either a cold
point or a chilly point.

Assume first that $P$ is a cold point for $\alpha$. By property (P.4)
and \cite[Coro.$\;$0.7]{Salt07}, the residual class
$\beta_{C_i\,,\,M}$ of $\alpha$ at any $C_i$ with respect to
$M=F((f\pi)^{1/q})$ is given by the class of a cyclic algebra $
(\chi_i\,,\,\bar{f}^{-t})$, where
\[
\chi_i=(L_i/\kappa(C_i)\,,\,\sigma_i)=\ram_{C_i}(\alpha)\in H^1(\kappa(C_i)\,,\,\mathbb{Z}/q)\,,
\]
$t$ is an integer prime to $q$ and $\ov{f}$ denotes the canonical
image of $f$ in $\kappa(C_i)$. Since $f$ is a $q$-th power in
$\kappa(P)$ by the choice, it follows easily that
$\beta_{C_i\,,\,M}$ is unramified at $P$. In the local ring
$R_P=\mathscr{O}_{X,\,P}$, we have $f\pi=u_P\pi_1^{s_1}\pi_2^{s_2}$ for some
$u_P\in R^*_P$ by \eqref{eq2temp} and property (P.6).  Hence, by
\cite[Prop.$\;$3.10 (c)]{Salt07}, $M$ splits all ramification of
$\alpha$ at $v$.

Finally,  consider the case where $P$ is a chilly point. Let
$C_1,\,C_2\in\set{C_i}$ be the two different irreducible components
of $\mathrm{Ram}_{X}(\alpha)$ through $P$ and let $\pi_1,\,\pi_2\in
R_P=\mathscr{O}_{X\,,\,P}$ be uniformizers of $C_1,\,C_2$ at $P$. Again by
\eqref{eq2temp} and property (P.6), we have
$f\pi=u_P\pi_1^{s_1}\pi_2^{s_2}$ for some $u_P\in R^*_P$. Let
$s=s(C_2/C_1)$ be the coefficient of $P$ with respect to $C_1$.
Using property (P.2), we find that $M=F((f\pi)^{1/q})$ may be
written in the form $M=F((\pi'_1\pi_2^{s})^{1/q})$, where $\pi'_1\in
R_P$ is a uniformizer of $C_1$ at $P$. Thus, by \cite[Prop.$\;$3.9
(a)]{Salt07}, $M/F$ splits all ramification of $\alpha$ at $v$, which
completes the proof.
\end{proof}

\begin{coro}\label{coro3p15temp}
Let $A$ be a $($noetherian$)$ henselian local domain with residue
field $\kappa$, $q$ a prime number unequal to the characteristic of
$\kappa$, and $X$ a regular excellent surface equipped with a proper
dominant morphism $X\to\Spec A$ whose closed fiber is of dimension $\le 1$. Let $F$ be the function field of
$X$ and $\alpha\in\Br(F)[q]$. Assume that $\kappa$ is a $B_1$ field and
that $\alpha$ has index $q$.

If either $\mu_q\subseteq F$ or $\kappa$ is perfect, then $\alpha$ is
represented by a cyclic algebra of degree $q$.
\end{coro}
\begin{proof}
If $\mu_q\subseteq F$, we may use \cite[Thm.$\;$7.13]{Salt08} to
find a degree $q$ Kummer extension $M/F=F(\sqrt[q]{g})/F$ that
splits all ramification of $\alpha$ over $X$. If $\kappa$  is perfect,
such an extension exists by Prop.$\;$\ref{prop3p14temp}. As in the
proof of Thm.$\;$\ref{thm3p6temp}, we have $\Br_{nr}(M/X)=0$ by
Coro.$\;$\ref{coro2p5temp}. Hence, $\alpha_M=0\in\Br(M)$. Then by a
theorem of Albert (cf. \cite[Prop.$\;$0.1]{Salt07}), which is rather
immediate when assuming the existence of a primitive $q$-th root of
unity, $\alpha$ is represented by a cyclic algebra of degree $q$.
\end{proof}

Recall that $R$ always denotes a 2-dimensional, henselian, excellent
local domain with fraction field $K$ and residue field $k$. Applying Coro.$\;$\ref{coro3p15temp} to a regular proper model
$\mathcal{X}\to\Spec R$ yields the folllowing.

\begin{thm}\label{thm3p16temp}
Assume that the residue field $k$ of $R$ has property $B_1$. Let $q$
be a prime number unequal to the characteristic of $k$.

If either $\mu_q\subseteq R$ or $k$ is perfect, then any Brauer class $\alpha\in\Br(K)[q]$
of index $q$ is represented by a cyclic algebra of degree $q$.
\end{thm}

\begin{remark}\label{remark3p17temp}
(1) In Prop.$\;$\ref{prop3p14temp} or Coro.$\;$\ref{coro3p15temp},
according to the above proof, if we assume the morphism $X\to\Spec
A$ is chosen such that $\Ram_X(\alpha)$ is an snc divisor and that
$\alpha$ has no cool points or chilly loops on $X$, then the hypothesis
that $\alpha$ has index $q$ may be replaced by the weaker condition
that all the residual classes $\beta_C$ of $ \alpha$ at all components
$C$ of $\Ram_X(\alpha)$ are split by the ramification.

(2) Similarly, let $\mathcal{X}\to\Spec R$ be a regular proper model such
that $\Ram_{\mathcal{X}}(\alpha)$ is an snc divisor and that $\alpha\in\Br(K)[q]$
has no cool points or chilly loops on $\mathcal{X}$. Then the conclusion in
Thm.\ref{thm3p16temp} remains valid if
instead of assuming $\alpha$ has index $q$ we only require that all the
residual classes of $\alpha$ at all components of $\Ram_{\mathcal{X}}(\alpha)$ are
split by the ramification.

(3) In the context of Thm.$\;$\ref{thm3p16temp}, if $k$ is a
separably closed field, \cite[Thm.$\;$2.1]{CTOP} proved a stronger
result: any Brauer class $\alpha\in\Br(K)$ of order $n$ which is
invertible in $R$ (but not necessarily a prime) is represented by a
cyclic algebra of index $n$.
\end{remark}

\subsection{Some corollaries}\label{sec3p4}

As applications of results obtained previously, we give a criterion for $\alpha\in\Br(K)[q]$ to have index $q$. Also, we will prove Thm.$\;$\ref{thm1p3temp}.

\

We begin with the following easy and standard fact.

\begin{lemma}\label{lemma3p18temp}
Let $R$ be a $2$-dimensional, henselian, excellent local domain with fraction field $K$ and residue field $k$. Let $\mathcal{X}\to\Spec R$
be a regular proper model. Then for any curve $C\subseteq \mathcal{X}$, one has either

$(\mathrm{i})$ $C$ is a proper curve over $k$;

or

$(\mathrm{ii})$ $C=\Spec B$, where $B$ is a domain whose normalization $B'$ is a henselian discrete valuation ring with residue field finite over $k$.
\end{lemma}
\begin{proof}After replacing $R$ by its normalization, we may assume $R$ is normal.

Consider the scheme-theoretic image $D$ of $C\subseteq \mathcal{X}$ by
the structural morphism $\mathcal{X}\to\Spec R$. If $D$ is the closed point
of $\Spec R$, then $C$ is a proper curve over the residue field
$k$.
 Otherwise, $D$ is the closed subscheme of $\Spec
R$ defined by a height 1 prime ideal $\mathfrak{p}\subseteq R$. Since $R$ is
2-dimensional and normal, the proper birational morphism
$\mathcal{X}\to\Spec R$ is an isomorphism over codimension 1 points (cf.
\cite[p.150, Coro.$\;$4.4.3]{Liu}). Thus the induced morphism $C\to D$ is
proper birational and quasi-finite, and hence by Chevalley's
theorem, it is finite. Write $A=R/\mathfrak{p}$ so that $D=\Spec A$. Then
$C=\Spec B$ for some domain $B\subseteq\kappa(C)=\kappa(D)$ which is
finite over $A$. Since $A$ is a henselian excellent local domain, the same is true for $B$. The normalization $B'$ of $B$ is finite over $B$ and hence
a henselian local domain as well, and it coincides with the normalization
of $A$ in its fraction field $\mathrm{Frac}(A)=\kappa(C)=\kappa(D)$. Then it is clear that
$B'$ is a henselian discrete valuation ring with residue
field finite over $k$. This finishes the proof.
\end{proof}

Recall that $\Omega_R$ denotes the set of discrete valuations of $K$ that are centered at codimension 1 points of regular proper models.

\begin{coro}\label{coro3p19temp}
For any $v\in\Omega_R$, the residue field $\kappa(v)$ is either the function field of a curve over $k$ or the fraction field
of a henselian discrete valuation ring whose residue field is finite over $k$.
\end{coro}

Now we can prove the following variant of \cite[Coro.$\;$5.2]{Salt07}.

\begin{coro}\label{coro3p20temp}
Let $q$ be a prime number unequal to the characteristic of the
residue field $k$ and $\alpha\in\Br(K)[q]$ a Brauer class of order $q$.
Let $\mathcal{X}\to\Spec R$ be a regular proper model such that the
ramification divisor $\Ram_{\mathcal{X}}(\alpha)$ of $\alpha$ on $\mathcal{X}$ has only
simple normal crossings and that $\alpha$ has no cool points or chilly
loops on $\mathcal{X}$. Write
$\ram_{C_i}(\alpha)=(L_i/\kappa(C_i)\,,\,\sigma_i)$ for the
ramification data and $\beta_i\in\Br(\kappa(C_i))$ for the residual classes.

Suppose that $k$ is a finite field. Then the following conditions
are equivalent:

$(\mathrm{i})$ $\alpha$ has index $q$.

$(\mathrm{ii})$ $\beta_i\in\Br(\kappa(C_i))$ is split by $L_i/\kappa(C_i)$
for every $i$.

$(\mathrm{iii})$ There are no hot points for $\alpha$ on $\mathcal{X}$.
\end{coro}
\begin{proof}We have
(i)$\Rightarrow$(ii)$\Rightarrow$(iii) by \cite[Prop.$\;$0.5 and Prop.$\;$3.10 (d)]{Salt07}.

To see (ii)$\Rightarrow$(i), note that by
Prop.$\;$\ref{prop3p14temp} and Remark$\;$\ref{remark3p17temp} that
there is a degree $q$ Kummer extension $M/K=K(\sqrt[q]{g})/K$ that
splits all the ramification of $\alpha$ over $R$. As the residue field
is finite, we have $\Br_{nr}(M/X)=0$ and in particular $\alpha_M=0$.
Hence, the index of $\alpha$ divides $q$, the degree of the extension
$M/K$. Since $\alpha$ has order $q$, it follows that the index of $\alpha$
is $q$.

To show (iii)$\Rightarrow$(ii), let $C$ be a fixed irreducible
component of the ramification divisor $\Ram_{\mathcal{X}}(\alpha)$ with
associated ramification data $\ram_C(\alpha)=(L/\kappa(C)\,,\,\sigma)$.
By  \cite[Lemma$\;$4.1]{Salt07}, there exists $\pi\in K^*$ having
the following properties: (1) the valuation $s_i:=v_{C_i}(\pi)$ with
respect to every component $C_i$ of $\Ram_{\mathcal{X}}(\alpha)$ is prime to
$q$; and (2) whenever there is a chilly point $P$ in the
intersection of two components $C_i$ and $C_j$, the coefficient of
$P$ with respect to $C_i$ is equal to
$s_j/s_i=v_{C_j}(\pi)/v_{C_i}(\pi)$ mod $q$.

Put $M:=K(\sqrt[q]{\pi})$. Let $\beta$ denote the residual class of
$\alpha$ with respect to $M$, i.e., the specialization of
$\alpha_M\in\Br(M)$ in $\Br(\kappa(C))$. We want to show that $\beta$
is split by $L/\kappa(C)$.

By Coro.$\;$\ref{coro3p19temp}, $\kappa(C)$ is either a function
field in one variable over the finite field $k$ or the fraction
field of a henselian discrete valuation ring with finite residue
field. The same is true for $L$. So, in either case,
$\beta\in\Br(\kappa(C))[q]$ is split by $L/\kappa(C)$ if and only if
$L/\kappa(C)$ splits all ramification of $\beta$ at every closed
point $P$ of $C$.

Assume first that $L/\kappa(C)$ is split at $P$. Then $P$ is either
a chilly point or a curve point ($P$ is not cold because
$L/\kappa(C)$ is unramified at $P$, cf.
Definition$\;$\ref{defn3p7temp}). If $P$ is chilly, $\beta$ is
unramified at $P$ by \cite[Prop.$\;$3.10 (b)]{Salt07}. If $P$ is a
curve point, then we conclude by \cite[Prop.$\;$3.11]{Salt07}.

Next consider the case where $L/\kappa(C)$ is nonsplit at $P$. Then
the $P$-adic valuation $v_P$ of $\kappa(C)$ extends uniquely to a
discrete valuation $w_P$ of $L$. If $L/\kappa(C)$ is ramified at
$P$, it is obvious that $L/\kappa(C)$ splits the ramification of
$\beta$ at $P$. If $L/\kappa(C)$ is unramified at $P$, then
$\kappa(w_P)$ is the unique degree $q$ extension of the finite field
$\kappa(v_P)=\kappa(P)$. Thus, the restriction map
\[
\mathrm{Res}\,:\;\;H^1(\kappa(P)\,,\,\mathbb{Z}/q\mathbb{Z})\lra
H^1(\kappa(w_P)\,,\,\mathbb{Z}/q\mathbb{Z})
\]is the zero map, which implies that $L/\kappa(C)$ splits the
ramification of $\beta$ at $P$. The corollary is thus proved.
\end{proof}

\begin{example}\label{exam3p21}
Here is a concrete example which shows the bound on the period-index exponent in Thm.$\;$\ref{thm3p6temp} is sharp. The criterion in the above corollary will be used in the argument.

Let $p$ be a prime number such that $p\equiv 3\pmod{4}$. Let
$k=\mathbb{F}_p$ be the finite field of cardinality $p$ and let
$R=k[\![x,\, y]\!]$ be the ring of formal power series in two
variables $x,\,y$ over $k$. Let $\mathcal{X}\to\Spec R$ be the blow-up of $\Spec R$ at the closed point and let $E\subseteq \mathcal{X}$
be the exceptional divisor. We have
\[
\Proj(k[T,\,S])\cong E=\Proj(R[T,\,S]/(x,\,y))\;\subseteq\;\mathcal{X}=\Proj\left(R[T,\,S]/(xS-yT)\right)\,.
\]
Let $f_1=y,\,f_2=x,\,f_3=y+x$ and let $C_i\subseteq\mathcal{X}$ be the strict transform of the curve
defined by $f_i=0$ in $\Spec R$ for each $i=1,\,2,\,3$. Each intersection $C_i\cap E$ consists of a single point $P_i$.

Let $\alpha$ be the Brauer class of the biquaternion algebra $(-1,\,y)\otimes(y+x,\,x)$ over $K=k(\!(x,\,y)\!)$. The ramification divisor
$\mathrm{Ram}_{\mathcal{X}}(\alpha)$ is $C_1+C_2+C_3+E$. The set of nodal points for $\alpha$ on $\mathcal{X}$ is $\{P_1,\,P_2,\,P_3\}$.
Locally at $P_1$, we may choose $s,\,x$ as local equations for $C_1,\,E$ respectively, where $s=S/T=y/x\in K$. Thus, in the Brauer group $\Br(K)$ we have
\[
\alpha=(-1,\,y)+(y+x,\,x)=(-1,\,xs)+(xs+x\,,\,x)=(-1,\,s)+(s+1,\,x)\,.
\]The function $s$ vanishes at $P_1$ and $-1\neq 1\in \kappa(P_1)^*/\kappa(P_1)^{*2}=k^*/k^{*2}$, so $P_1$ is a hot point by definition (cf. Remark$\;$\ref{remark3p8temp}). One may verify that $P_2$ and $P_3$ are cold points.

As for the residual classes, one may check that for each $i$ the residual classes of $\alpha$ at $C_i$ are split by the ramification. Let us now
show that at $E$ the residual classes are not split by the ramification. Indeed, if $v_E$ denotes the discrete valuation of $K$ defined by $E$, we have
$v_E(x)=v_E(y)=v_E(y+x)=1$. Then it is easy to see that the ramification $\mathrm{ram}_E(\alpha)$ of $\alpha$ at $E$ is represented by the
quadratic extension $k(s)(\sqrt{s+1})$ of $\kappa(E)=k(s)$. Putting $M=K(\sqrt{x})$, $\alpha_M=(-1,\,y)=(-1,\,s)\in \Br(M)$ and hence the residue
class of $\alpha$ at $E$ with respect to $M/K$ is $\beta_E=(-1,\,s)\in\Br(\kappa(E))=\Br(k(s))$. Putting $u=\sqrt{s+1}$, it is
easy to see that the quaternion algebra $(-1,\,s)=(-1,\,u^2-1)$ is not split over $k(u)=k(s)(\sqrt{s+1})$ (in fact, it
ramifies at $u=1$ since $-1$ is not a square in $k=\mathbb{F}_p$).

By Thm.$\;$\ref{thm3p6temp} and Coro.$\;$\ref{coro3p20temp}, we conclude that  $\alpha\in \Br(K)[2]$ is of index 4.
\end{example}

We shall now prove Thm.$\;$\ref{thm1p3temp} in a slightly
generalized form.

\begin{thm}\label{thm3p21temp}
Let $R$ be a $2$-dimensional, henselian, excellent local domain
whose residue field $k$ has property $B_1$, $K$ the fraction field
of $R$, $q$ a prime number unequal to the characteristic of $k$, and
$\alpha\in\Br(K)[q]$. Assume either $\mu_q\subseteq R$ or $k$ is
perfect.

If for every $v$ in the set $\Omega_R$ of discrete valuations of $K$
that correspond to codimension one points of regular proper models,
the Brauer class $\alpha_v=\alpha\otimes_KK_v\in\Br(K_v)$ is represented
by a cyclic algebra of degree $q$ over $K_v$, then $\alpha$ is
represented by a cyclic algebra of degree $q$ over $K$.
\end{thm}
\begin{proof}
Let $\mathcal{X}\to \Spec R$ be a regular proper model such that
$\Ram_{\mathcal{X}}(\alpha)$ is an snc divisor and that $\alpha$ has no cool
points or chilly loops on $\mathcal{X}$. By Thm.$\;$\ref{thm3p16temp} and
Remark$\;$\ref{remark3p17temp} (2), it suffices to prove that all
the residual classes of $\alpha$ at all components of $\Ram_{\mathcal{X}}(\alpha)$
are split by the ramification.

Assume the contrary. Then there is an irreducible component  $C$ of
$\Ram_{\mathcal{X}}(\alpha)$ such that the residual classes of $\alpha$ at $C$ are
not split by the ramification. Now consider the discrete valuation
$v=v_C$ of $K$ defined by $C$. By assumption, $\alpha_v\in\Br(K_v)[q]$
is cyclic of degree $q$, so that $\alpha_v=(\chi_v\,,\,b_v)$ for some
$\chi_v\in H^1(K_v\,,\,\mathbb{Z}/q)$ and $b_v\in K_v^*$. Without loss of
generality, we may assume $b_v=w\pi^t$, where $\pi\in K^*_v$ is a
uniformizer for $v$, $w\in K_v^*$ is a unit for $v$ and $t$ is an
integer such that $0\le t\le q-1$.

Let $(L/K_v\,,\,\sigma_v)$ be the pair representing the character
$\chi_v\in\Hom_{cts}(G_{K_v}\,,\,\mathbb{Z}/q)$. If $L/K_v$ is unramified,
then $t\neq 0$, because $\alpha$ is ramified at $v=v_C$. Then there exist integers $r,\,s\in \mathbb{Z}$ such that $1=rq+st$. Putting $\pi'=w^s\pi$, we have $b_v=w\pi^t=(w^r)^q(\pi')^t$.  Then
\[
\alpha_v=(\chi_v\,,\,b_v)=(\chi_v\,,\,(\pi')^t)\in \Br(K_v)[q]\,
\]is clearly split by the totally ramified
extension $K'_v/K_v:=K_v(\sqrt[q]{\pi'})/K_v$. In particular, the
residual class of $\alpha$ with respect to $K'_v/K_v$ is 0 and a
fortiori the residual classes of $\alpha$ at $C$ are split by the
ramification. But this contradicts our choice of $C$.

If $L/K_v$ is ramified, then it is totally and tamely ramified. So $L=K_v(\sqrt[q]{\theta})$ for some $\theta\in K_v$. The extension being Galois,
it follows that $\mu_q\subseteq L$. Since the residue fields of $L$ and $K_v$ are the same, Hensel's lemma implies that $\mu_q\subseteq K_v$.
We may thus assume $\alpha_v=(u\pi^s\,,\,b_v)=(u\pi^s\,,\,w\pi^t)$ for
some unit $u\in K_v^*$ and some integer $s$ such that $0\le s\le
q-1$. Since $\alpha$ is ramified at $v=v_C$, $s$ and $t$ cannot be both
0. Assume for instance $s>0$. A similar argument as before shows
that $\alpha_v$ is split by a totally ramified extension
$K_v(\sqrt[q]{\pi})/K_v$, which leads to a contradiction again. This
proves the theorem.
\end{proof}

The following special case, which answers a question in \cite[Remark$\;$3.7]{CTOP}, will be used in section$\;$\ref{sec4}.

\begin{coro}\label{coro3p22temp}
Let $R$ be a $2$-dimensional, henselian, excellent local domain with
fraction field $K$ and residue field $k$. Assume that $k$ is a
finite field of characteristic $\neq 2$. Let $D$ be a central simple
algebra  over $K$ of period  $2$.

If for every $v\in\Omega_R$, $D\otimes_KK_v$ is Brauer equivalent to
a quaternion algebra over $K_v$, then $D$ is Brauer equivalent to a
quaternion algebra over $K$.
\end{coro}

The analog of Thm.$\;$\ref{thm3p21temp} in the case of a $p$-adic
function field does not seem to have been noticed. Let us prove it
in the following general form.

\begin{thm}\label{thm3p23temp}
Let $A$ be a henselian, excellent discrete valuation ring whose
residue field $\kappa$ has property $B_1$. Let $q$ be a prime number
unequal to the characteristic of $\kappa$, $F$ the function field of
an algebraic curve over the fraction field of $A$ and $\alpha\in
\Br(F)[q]$.  For every regular surface $\mathcal{Y}$ with function field $F$
equipped with a proper flat morphism $\mathcal{Y}\to\Spec A$, let
$\Omega(F/\mathcal{Y}^{(1)})$ denote the set of discrete valuations of $F$
corresponding to codimension one points of $\mathcal{Y}$. Let $\Omega_{A,\,F}$ be
the union of all $\Omega(F/\mathcal{Y}^{(1)})$, where $\mathcal{Y}$ runs over regular surfaces as above.

Assume either $\mu_q\subseteq F$ or $\kappa$ is perfect. If for
every $v\in\Omega_{A,\,F}$, $\alpha_v=\alpha\otimes_FF_v\in\Br(F_v)$ is
represented by a cyclic algebra of degree $q$ over $F_v$, then $\alpha$
is represented by a cyclic algebra of degree $q$ over $F$.
\end{thm}
\begin{proof}
By resolution of singularities, there exists a regular surface $X$
with function field $F$, together with a proper flat morphism
$X\to\Spec A$, such that the ramifcation divisor $\Ram_{X}(\alpha)$ has
only simple normal crossings and $\alpha$ has no cool points or chilly
loops on $X$. By Coro.$\;$\ref{coro3p15temp} and
Remark$\;$\ref{remark3p17temp}, it suffices to show that all the
residual classes of $\alpha$ are split by the ramification. This may be
done as in the proof of Thm.$\;$\ref{thm3p21temp}.
\end{proof}

\section{Quadratic forms and the $u$-invariant}\label{sec4}

\subsection{A local-global principle from class field theory}\label{sec4p1}

To prove our results on quadratic forms, a result coming from Saito's
work on class field theory for 2-dimensional local rings will be used at a few crucial points.

\

As before, let $R$ be a 2-dimensional, henselian, excellent local
domain with fraction field $K$ and residue field $k$. The following
proposition grows out of a conversation with S. Saito.

\begin{prop}\label{prop4p1temp}
Assume $R$ is normal and the residue field $k$ is finite. Let
$\mathcal{X}\to\Spec R$ be a regular proper model such that the reduced
divisor on the closed fiber has only simple normal crossings. Let
$n>0$ be an integer that is invertible in $k$.

Then the natural map
\[
H^3(K\,,\,\mu_n^{\otimes 2})\lra \prod_{v\in\mathcal{X}^{(1)}}H^3(K_v\,,\,\mu_n^{\otimes 2})
\]is injective.
\end{prop}
\begin{proof}
Let $Y$ be the reduced subscheme on the closed fiber of $\mathcal{X}\to\Spec
R$ and $U=\mathcal{X}\setminus Y$. Let $i: Y\to \mathcal{X}$ and $j: U\to\mathcal{X}$ denote
the natural inclusions and put $\mathscr{F}:=i^*Rj_*\mu_n^{\otimes 2}$. Let
$P=(\Spec R)^{(1)}$ be the set of codimension 1 points of $\Spec R$.
We may identify $P$ with the set of closed points of $U$ via the
structural map $\mathcal{X}\to\Spec R$. From localization theories we obtain
exact sequences (cf. \cite[p.358--360]{Sai87})
\begin{equation}\label{eq4p1temp}
H^3(U\,,\,\mu_n^{\otimes 2})\overset{\phi}{\lra} H^3(K\,,\,\mu_n^{\otimes
2})\overset{\iota}{\lra} \bigoplus_{\mathfrak{p}\in
P}H^3(K_{\mathfrak{p}}\,,\,\mu_n^{\otimes 2}) \end{equation}and
\begin{equation}\label{eq4p2temp}
H^3(Y\,,\,\mathscr{F})\lra \bigoplus_{\eta\in
Y^{(0)}}H^3(\kappa(\eta)\,,\,\mathscr{F})\overset{\theta'}{\lra}
\bigoplus_{x\in Y^{(1)}}\mathbb{Z}/n\end{equation}where the map $\iota$ in
\eqref{eq4p1temp} is induced by the natural maps
$H^3(K\,,\,\mu_n^{\otimes 2})\to H^3(K_{\mathfrak{p}}\,,\,\mu_n^{\otimes
2})$. For each $\eta\in Y^{(0)}\subseteq \mathcal{X}^{(1)}$, let $A_{\eta}$
be the completion of the discrete valuation ring
$\mathscr{O}_{\mathcal{X}\,,\,\eta}$. By the functoriality of the functor $Rj_*$, we
have the following commutative diagram
\[\begin{CD}
H^3(\mathcal{X}\,,\,Rj_*\mu_n^{\otimes 2})@>>>
H^3(\mathscr{O}_{\mathcal{X},\,\eta}\,,\,Rj_*\mu_n^{\otimes 2}) @>>> H^3(A_{\eta}\,,\,Rj_*\mu_n^{\otimes 2})\\
@V{\cong}VV  @VV{\cong}V @VV{\cong}V\\
H^3(U\,,\,\mu_n^{\otimes 2}) @>>> H^3(K\,,\,\mu_n^{\otimes 2}) @>>>
H^3(K_{\eta}\,,\,\mu_n^{\otimes 2})
\end{CD}\]
 where the vertical maps are
canonical isomorphisms. On the other hand, we have a commutative
diagram
\[\begin{CD}
H^3(\mathcal{X}\,,\,Rj_*\mu_n^{\otimes 2})@>>>
H^3(\mathscr{O}_{\mathcal{X},\,\eta}\,,\,Rj_*\mu_n^{\otimes 2}) @>>> H^3(A_{\eta}\,,\,Rj_*\mu_n^{\otimes 2})\\
@V{\cong}VV  @VVV @VV{\cong}V\\
H^3(Y\,,\,\mathscr{F}) @>>> H^3(\kappa(\eta)\,,\,\mathscr{F}) @>{\mathrm{id}}>>
H^3(\kappa(\eta)\,,\,\mathscr{F})
\end{CD}\]where the two vertical isomorphisms come from proper base
change. Let
\[
\theta: \bigoplus_{\eta\in Y^{(0)}}H^3(K_{\eta}\,,\,\mu_n^{\otimes
2})\to \bigoplus_{x\in Y^{(1)}}\mathbb{Z}/n\] be the composition of the map
$\theta'$ in \eqref{eq4p2temp} with the canonical isomorphism
\[
\bigoplus_{\eta\in Y^{(0)}}H^3(K_{\eta}\,,\,\mu_n^{\otimes 2})\cong
\bigoplus_{\eta\in Y^{(0)}}H^3(\kappa(\eta)\,,\,\mathscr{F})\,.
\]Putting all things together, we get a commutative diagram with
exact rows
\[
\begin{CD}
H^3(U\,,\,\mu_n^{\otimes 2}) @>\phi>> H^3(K\,,\,\mu_n^{\otimes
2})@>\iota>> \bigoplus_{\mathfrak{p}\in
P}H^3(K_{\mathfrak{p}}\,,\,\mu_n^{\otimes 2})\\
@V{\cong}VV @V{\varphi}VV \\
H^3(Y\,,\,\mathscr{F}) @>\tau>> \bigoplus_{\eta\in
Y^{(0)}}H^3(K_{\eta}\,,\,\mu_n^{\otimes 2}) @>\theta>>
\bigoplus_{x\in Y^{(1)}}\mathbb{Z}/n
\end{CD}
\]where the map $\varphi$ is induced by the restriction
maps $H^3(K\,,\,\mu_n^{\otimes 2})\to
H^3(K_{\eta}\,,\,\mu_n^{\otimes 2})$. By \cite[p.361,
Lemma$\;$5.13]{Sai87}, the induced map $\mathrm{Ker}(\phi)\to\mathrm{Ker}(\tau)$ is
an isomorphism. Hence, $\varphi$ induces an isomorphism
$\mathrm{Ker}(\iota)\cong \mathrm{Ker}(\theta)$. In particular, an element $\zeta\in
H^3(K\,,\,\mu_n^{\otimes 2})$ vanishes if and only if
$\iota(\zeta)=0=\varphi(\zeta)$. The result then follows immediately
since $\mathcal{X}^{(1)}=P\cup Y^{(0)}$.
\end{proof}

An application of Prop.$\;$\ref{prop4p1temp} that we will need is the following variant of \cite[Thm.$\;$3.5]{PS10} (see also \cite[Coro.$\;$4.3]{PS10b}).

\begin{thm}\label{thm4p2NEW}
Let $R$ be a $2$-dimensional, henselian, excellent local domain with finite residue field $k$ and fraction field $K$. Let $q$ be a prime number that is different from the characteristic of $k$. If $\mu_q\subseteq K$, then every element in $H^3(K\,,\,\mu_q^{\otimes 3})$ is a symbol.
\end{thm}
\begin{proof}
First note that as in \cite[Coro.$\;$1.2]{PS10}, one can show that for any $v\in\Omega_R$, the kernel
$H^3_{nr}(K_v\,,\,\mu_q^{\otimes 3})$ of the residue map
\[
H^3(K_v\,,\,\mu_q^{\otimes 3})\overset{\partial_v}{\lra}
H^2(\kappa(v)\,,\,\mu_q^{\otimes 2})
\]is trivial.  Indeed, by Coro.$\;$\ref{coro3p19temp}, the residue field $\kappa(v)$ is
either a function field in one variable over $k$ or the fraction
field of a henselian discrete valuation ring $B$ whose residue field
$k'$ is finite over $k$. In the former case, we have
$\mathrm{cd}_q(\kappa(v))\le 2$ by \cite[p.93, Prop.$\;$11]{Ser2}
and hence $H^3(\kappa(v)\,,\,\mu_q^{\otimes 3})=0$. In the latter
case, the exact sequence
\[
H^3(k'\,,\,\mu_q^{\otimes 3})\cong H^3_{\et}(B\,,\,\mu_q^{\otimes
3})\lra H^3(\kappa(v)\,,\,\mu_q^{\otimes 3})\lra
H^2(k'\,,\,\mu_q^{\otimes 2})
\]implies $H^3(\kappa(v)\,,\,\mu_q^{\otimes 3})=0$, since $\mathrm{cd}(k')\le 1$.

Let $\mathcal{X}\to\Spec R$ be a regular proper model. Since $H^3_{nr}(K_v,\,\mu_q^{\otimes 3})=0$ for all $v\in\Omega_R$, it follows that the
kernel of the natural map
\[
H^3(K\,,\,\mu_q^{\otimes 3})\lra
\prod_{v\in\mathcal{X}^{(1)}}H^3(K_v\,,\,\mu_q^{\otimes 3})\,
\]coincides with
\[
H^3_{nr}(K/\mathcal{X}^{(1)}\,,\,\mu^{\otimes 3}_q):=\bigcap_{v\in\mathcal{X}^{(1)}}\mathrm{Ker}\left(\partial_v\,: H^3(K\,,\,\mu_q^{\otimes 3})\lra H^2(\kappa(v)\,,\,\mu_q^{\otimes 2})\right)\,.
\]As $\mu_q\subseteq K$, it follows from
Prop.$\;$\ref{prop4p1temp} that
$H^3_{nr}(K/\mathcal{X}^{(1)}\,,\,\mu^{\otimes 3}_q)=0$. On the other hand, Lemma$\;$\ref{lemma3p18temp} implies that $\mathcal{X}$
is a $q$-special surface in the sense of \cite[$\S$3]{PS10b}. So the result follows from \cite[Thm.$\;$4.2]{PS10b}.
\end{proof}

\subsection{Local-global principle for quadratic forms}

Thanks to Coro.$\;$\ref{coro3p22temp} we are now in a position to
prove the local-global principle for quadratic forms of
rank 5 with respect to the set $\Omega_R$ of divisorial valuations as asserted
in Thm.$\;$\ref{thm1p4temp}. Standard notations from the algebraic theory of quadratic forms (as in \cite{Lam} or
\cite{Sch85}) will be used as of now.

\begin{proof}[Proof of Thm.$\;\ref{thm1p4temp}$]
We follow the ideas in the proof of \cite[Thm.$\;$3.6]{CTOP}.
Let $\phi$ be a 5-dimensional nonsingular quadratic form over $K$.
Assume that $\phi_v$ is isotropic over $K_v$ for every
$v\in\Omega_R$.

The 6-dimensional form $\psi:=\phi\bot \langle -\det(\phi)\rangle$
is similar to a so-called Albert form $\langle
a\,,\,b\,,\,-ab\,,\,-c\,,\,-d\,,\,cd\rangle$. By the general theory of
Albert forms (cf. \cite[p.14, Thm.$\;$1.5.5]{GS}), the form $\langle
a\,,\,b\,,\,-ab\,,\,-c\,,\,-d\,,\,cd\rangle$ is isotropic if and
only if the biquaternion algebra $D:=(a\,,\,b)\otimes(c\,,\,d)$ is
not a division algebra. By assumption, for every $v\in\Omega_R$,
$\psi_v$ is isotropic over $K_v$, so the biquaternion algebra
$D_v=(a\,,\,b)_{K_v}\otimes(c\,,\,d)_{K_v}$ is not a division
algebra. The index of $D_v$ must be smaller than 4, which is the
degree. Therefore, $D_v$ is Brauer equivalent to a quaternion
algebra over $K_v$. By Coro.$\;$\ref{coro3p22temp},
$D=(a\,,\,b)\otimes (c\,,\,d)$ is Brauer equivalent to a quaternion
algebra over $K$. In particular, $D$ is not a division algebra over
$K$. Hence $\psi$ is isotropic over $K$. This implies that $\phi$
may be written in the form
\[
\phi=\det(\phi). \langle 1\,,\,a\,,\,b\,,\,c\,,\,abc\rangle
\]over $K$. In particular, $\phi$ is similar to a subform of the 3-fold Pfister form
\[
\rho=\langle\!\langle\,-a\,,\,-b\,,\,-c\rangle\!\rangle:=\langle
1\,,\,a\rangle\otimes \langle 1\,,\,b\rangle\otimes \langle
1\,,\,c\rangle\,.
\]By Merkurjev's theorem (cf. \cite[p.129, Prop.$\;$2]{Ara84}), the form $\rho$ is
isotropic if and only if the symbol class
$(-a\,,\,-b\,,\,-c)$ vanishes. For each $v\in\Omega_R$, as the
subform $\phi_v$ of $\rho_v$ is isotropic over $K_v$, we have
$(-a\,,\,-b\,,\,-c)=0$ in $H^3(K_v\,,\,\mathbb{Z}/2)$. Then it follows from
Prop.$\;$\ref{prop4p1temp} that $(-a\,,\,-b\,,\,-c)=0$ in
$H^3(K\,,\,\mathbb{Z}/2)$ (noticing that we may assume $R$ is normal). Thus,
the Pfister form $\rho$ is isotropic over $K$ and hence hyperbolic
(\cite[p.319, Thm.$\;$1.7]{Lam}). The form $\rho$ then contains a
4-dimensional totally isotropic subspace, which must intersect the
underlying space of the 5-dimensional subform $\phi$ in a nontrivial
subspace. Hence, $\phi$ is isotropic over $K$.
\end{proof}

\begin{remark}\label{remark4p2temp}
Let $R$ be a $2$-dimensional, henselian, excellent local domain with
fraction field $K$ and residue field $k$. Assume the characteristic
of $k$ is not 2. We record results and open questions on
local-global principle for isotropy of quadratic forms over $K$ as
far as we know.

(1) When the residue field $k$ is finite, Thm.$\;$\ref{thm1p4temp}
establishes the local-global principle with respect to discrete
valuations in $\Omega_R$ only for rank 5 forms. If $R$ is of the
form $R=A[\![t]\!]$, where $A$ is a complete discrete valuation
ring, then the same local-global principle is proved for quadratic
forms of any rank $\ge 5$ in \cite[Thm.$\;$1.2]{Hu10}. There the
residue field may be any $C_1$-field.

(2) For general $R$ with finite residue field, it remains an open
question whether the local-global principle holds for quadratic
forms of rank 6, 7 or 8.

(3) Generalizing an earlier result of \cite{CTOP},
\cite[Thm.$\;$1.1]{Hu10} proves the local-global principle for forms
of rank 3 or 4 when the residue field $k$ is arbitrary (not
necessarily finite, $C_1$ or $B_1$).

(4) The above results do not extend to binary forms even if $k$ is
finite. For example, Jaworski \cite[Thm.$\;$1.5]{Ja} shows that if
$K$ is the fraction field of the ring
\[
R=k[\![x,\,y,\,z]\!]/(z^2-(y^2-x^3)(x^2-y^3))\,,
\]then $y^2-x^3$ is a square in $K_v$ for every discrete valuation $v$ of
$K$, but it is not a quare in $K$.
\end{remark}

\subsection{The $u$-invariant}\label{sec4p2}

Let $F$ be a field of characteristic $\neq 2$. We will denote by
$W(F)$ the Witt ring of quadratic forms over $F$ and by $I(F)$ the
fundamental ideal. For $a_1\,,\dotsc, a_n\in F^*$, $\langle\!\langle
a_1\,,\dotsc, a_n\rangle\!\rangle$ denotes the $n$-fold Pfister form
$\langle 1\,,\,-a_1\rangle\otimes\cdots\otimes \langle
1\,,\,-a_n\rangle$. The $n$-th power $I^n(F)$ of the fundamental
ideal $I(F)$ is generated by the $n$-fold Pfister forms. Recall that the
$u$-\empha{invariant} of $F$, denoted $u(F)$, is defined as the supremum of dimensions of anisotropic quadratic forms over $F$ (so $u(F)=\infty$, if
such dimensions can be arbitrarily large).

\

Let $R$ be a $2$-dimensional, henselian, excellent local domain with
fraction field $K$ and residue field $k$. Assume that the residue
field $k$ of $R$ is finite of characteristic $\neq 2$. Then the
inequality $u(K)\ge 8$ may be easily seen as follows: Take any
discrete valuation $v$ corresponding to a height 1 prime ideal of
the normalization of $R$. It follows from
Coro.$\;$\ref{coro3p19temp} and a theorem of Springer (\cite[p.209,
Coro.$\;$2.6]{Sch85}) that $u(\kappa(v))=4$. Take a 4-dimensional
diagonal form $\varphi$ over $K$ whose coefficients  are units for $v$
such that its residue form $\ov{\varphi}$ over $\kappa(v)$ is
anisotropic and let $\pi\in K$ be a uniformizer for $v$. Then
$\varphi\bot\pi.\varphi$ is an 8-dimensional form over $K$ which is
anisotropic over $K_v$.

The rest of this subsection is devoted to the proof of
Thm.$\;$\ref{thm1p5temp}, which asserts that if $k$ is a finite
field of characteristic $\neq 2$, then $u(K)\le 8$ (or equivalently
$u(K)=8$ according to the proceeding paragraph). Basically, we will
follow the method of Parimala and Suresh developed in \cite{PS10} (see also \cite[Appendix]{PS10b}).

\begin{prop}[{\cite[Prop.$\;$4.3]{PS10}}]\label{prop4p4NEW}
Let $F$ be a field of characteristic $\neq 2$. Assume the following properties hold:

$(\mathrm{i})$ $I^4(F)=0$.

$(\mathrm{ii})$ Every element of $I^3(F)$ is represented by a $3$-fold Pfister form.

$(\mathrm{iii})$ Every element of $H^2(F\,,\,\mathbb{Z}/2)$ is the sum of two symbols.

$(\mathrm{iv})$ If $\phi$ is a $3$-fold Pfister form and $\varphi_2$ is $2$-dimensional quadratic form over $F$, then there exist
$f,\,a,\,b\in F^*$ such that $f$ is a value of $\varphi_2$ and $\phi=\langle 1,\,f\rangle\otimes \langle 1,\,a\rangle\otimes\langle 1,\,b\rangle$ in
the Witt group $W(F)$.

$(\mathrm{v})$ If $\phi=\langle 1,\,f\rangle\otimes \langle 1,\,a\rangle\otimes\langle 1,\,b\rangle$ is a $3$-fold Pfister form and $\varphi_3$ is $3$-dimensional quadratic form over $F$, then there exist
$g,\,h\in F^*$ such that $g$ is a value of $\varphi_3$ and $\phi=\langle 1,\,f\rangle\otimes \langle 1,\,g\rangle\otimes\langle 1,\,h\rangle$ in
the Witt group $W(F)$.

Then $u(F)\le 8$.
\end{prop}

Property (i) in the above proposition is verified for the field $K$ by using the following deep well-known theorem.

\begin{thm}[{Artin--Gabber}]\label{thm4p6temp}
Let $R$ be a $2$-dimensional henselian excellent local domain with
fraction field $K$ and finite residue field $k$.

Then for every prime number $p$ different from the characteristic of
$k$, the $p$-cohomological dimension $\mathrm{cd}_p(K)$ of $K$ is
$3$.
\end{thm}
When $K$ and $k$ have the same characteristic, this follows from a
theorem of Artin (\cite[Exp.$\;$XIX, Coro.$\;$6.3]{SGA}). When the
characteristic of $K$ is different from that of $k$, Gabber proved
the analog of Artin's result. A different proof due to Kato may be given along the lines of the case treated in \cite[$\S$5]{Sai86}.

\begin{coro}\label{coro4p7temp}
Let $R$ be a $2$-dimensional, henselian, excellent local domain with
finite residue field of characteristic $\neq 2$. Let $K$ be the
fraction field of $R$.

Then $I^4(K)=0$.
\end{coro}
\begin{proof}
This follows by combining Thm.$\;$\ref{thm4p6temp} and a result of
Arason, Elman and Jacob (\cite[Coro.$\;$4]{AEJ86}).
\end{proof}

\begin{lemma}\label{lemma4p9temp}
Let $R$ be a $2$-dimensional, henselian, excellent local domain with
fraction field $K$ and residue field $k$.  Assume that $k$ is a
$B_1$ field of characteristic $\neq 2$.

Then every element in $H^2(K\,,\,\mathbb{Z}/2)$ is the sum of two symbols.
\end{lemma}
\begin{proof}
By Thm.$\;$\ref{thm3p6temp}, every element in $\Br(K)[2]\cong
H^2(K\,,\,\mathbb{Z}/2)$ has index dividing 4. A well-known theorem of
Albert (\cite[Chap.$\;$XI, $\S$6, Thm.$\;$9]{Alb}) then implies
that it is the class of a tensor product of two quaternion algebras.
\end{proof}

\begin{proof}[Proof of Thm.$\;\ref{thm1p5temp}$]
We have $I^3(K)\cong H^3(K\,,\,\mathbb{Z}/2)$ in view of
Coro.$\;$\ref{coro4p7temp} (\cite{AEJ86}). Thus every element of
$I^3(K)$ is represented by a 3-fold Pfister form by Thm.$\;$\ref{thm4p2NEW}. That the field $K$
has property (iii) in Prop.$\;$\ref{prop4p4NEW} is proved in Lemma$\;$\ref{lemma4p9temp}. Finally,
the same argument as in the proof of \cite[Appendix, Prop.$\;$3]{PS10b} proves that the field $K$ has properties (iv) and (v) in Prop.$\;$\ref{prop4p4NEW}.
The theorem is thus proved.
\end{proof}

\begin{question}[Suresh]\label{Ques4p11temp}
Let $R,\,K$ and $k$ be as usual and assume the residue field $k$ is an arbitrary (not necessarily finite) field of characteristic $\neq 2$.
It is known that $u(K)=4u(k)$ in each of the following special cases:

(1) $k$ is finite (Theorem$\;$\ref{thm1p5temp}).

(2) $k$ is hereditarily quadratically closed (i.e. every finite extension field of $k$ is quadratically closed). This basically follows from the proof of \cite[Thm.$\;$3.6]{CTOP}.

(3) Assume the residue field $k$ has the following property:

$(*)$ For each finitely generated field extension $L/k$ of transcendence degree $d\le 1$, $u(L)\le 2^d u(k)$.

Then Harbater, Hartmann and Krashen prove that $u(K)=4u(k)$ holds in the following cases:

(3.1) $R=A[\![t]\!]$, where $A$ is a complete discrete valuation ring (\cite[Coro.$\;$4.19]{HHK});

(3.2) $K$ is a finite separable extension of $k(\!(x,\,y)\!)$ (\cite[Coro.$\;4.2$]{HHK11b}).

Question: Is the relation $u(K)=4u(k)$ always true under hypothesis $(*)$? Is it true when assuming moreover $R$ is complete?
\end{question}

\subsection{Torsors under special orthogonal groups}\label{sec4p3}

\begin{thm}\label{thm4p11temp}
Let $R$ be a $2$-dimensional, henselian, excellent local domain with
fraction field $K$ and residue field $k$. Assume $k$ is a finite
field of characteristic $\neq 2$.

Then for any nonsingular quadratic form $\phi$ of rank $\ge 2$ over
$K$, the natural map
\[
H^1(K\,,\,\mathrm{SO}(\phi))\lra\prod_{v\in\Omega_R}H^1(K_v\,,\,\mathrm{SO}(\phi))
\]
is injective.
\end{thm}
\begin{proof}
Let $\psi\,,\,\psi'$ be nonsingular quadratic forms representing classes in
$H^1(K\,,\,\mathrm{SO}(\phi))$. As they have the same dimension, the
forms $\psi$ and $\psi'$ are isometric if and only if they represent
the same class in the Witt group. Since $\psi$ and $\psi'$ also have the
same discriminant, it follows from \cite[p.82, Chapt.$\;$2,
Lemma$\;$12.10]{Sch85} that $\psi-\psi'\in I^2(K)$. Now it suffices
to apply Lemma$\;$\ref{lemma4p12temp} below.
\end{proof}

\begin{lemma}\label{lemma4p12temp}
Let $R,\,K,\,k$ be as in Theorem$\;\ref{thm4p11temp}$. The natural
map
\[
I^2(K)\lra\prod_{v\in\Omega_R}I^2(K_v)
\]is injective.
\end{lemma}
\begin{proof}
In the case where $R$ is the henselization of an algebraic surface
over a finite field at a closed point, this is already established
in \cite[Thm.$\;$3.10]{CTOP}. Here the argument is essentially the
same, with Prop.$\;$\ref{prop4p1temp} and Coro.$\;$\ref{coro4p7temp}
substituting appropriate ingredients in that case.
\end{proof}

\begin{remark}\label{remark4p13temp}
In Thm.$\;$\ref{thm4p11temp}, if $\phi$ is of dimension 2 or 3, one need not assume the residue field $k$ finite.

Indeed, let $\psi$ and $\psi'$ be nonsingular forms representing classes in
$H^1(K\,,\,\mathrm{SO}(\phi))$. In the 2-dimensional case, assume
$\psi'\cong \langle a\,,\,b\rangle$ and
$\psi\cong\langle\alpha\,,\,\beta\rangle$. Then $\psi'\cong\psi$ if and
only if the quaternion algebras $(a,\,b)$ and $(\alpha\,,\,\beta)$ are
isomorphic, since the two forms have the same discriminant (cf.
\cite[Chapt.$\;$2, Coro.$\;$11.11]{Sch85}). In the 3-dimensional
case, assume $\psi'\cong \langle a\,,\,b\,,c\rangle$  and
$\psi\cong\langle \alpha\,,\,\beta\,,\,\gamma\rangle$. Then $\psi'\cong
\psi$ if and only if the quaternion algebras $(-ac\,,\,-bc)$ and
$(-\alpha\gamma\,,\,-\beta\gamma)$ are isomorphic (cf.
\cite[Chapt.$\;$2, Thm.$\;$11.12]{Sch85}). Since two quaternion
algebras are isomorphic if and only if their classes in the Brauer
group coincide, the result then follows from the injectivity of the natural map
\[
\Br(K)\lra\prod_{v\in\Omega_R}\Br(K_v)\,,
\]
this last local-global statement being essentially established in \cite[$\S$1]{CTOP} (see also the proof of \cite[Thm.$\;$1.1]{Hu10}).
\end{remark}

\begin{remark}\label{remarkQF3p4p16}
Let $F$ be a field of characteristic $\neq 2$ and $\Omega$ a set of
discrete valuations of $F$. For each integer $r\ge 2$, consider the
following statements:

($LG_r$) For any two nonsingular quadratic forms of rank $r$ that have the same discriminant over $F$, if they are isometric over $F_v$ for every $v\in\Omega$, then they must already be isometric over $F$.

($LG'_r$) If a nonsingular quadratic form of rank $r$ over $F$ is isotropic over $F_v$ for every $v\in\Omega$, then it is isotropic over $F$.

Theorem$\;$\ref{thm4p11temp} amounts to saying that for every $r\ge 2$, ($LG_r$) is true for the field $K$ with respect to its divisorial valuations. Our proof of this theorem does not rely on the local-global principle for the \textit{isotropy} of quadratic forms. Note however that over an \textit{arbitrary} field $F$ (of characteristic $\neq 2$) one has $(LG_r)+(LG'_{r+2})\Longrightarrow (LG_{r+1})$.

Indeed, let $\psi,\,\psi'$ be nonsingular quadratic forms of rank
$r+1$ over $F$ which have the same discriminant. Assume $\psi\cong \langle
a_1 \rangle\bot \psi_1$ with $\psi_1$ of rank $r$. If
$(\psi)_{F_v}\cong (\psi')_{F_v}$ for every $v\in \Omega$, then
$(\psi'\bot \langle -a_1 \rangle)_{F_v}$ is isotropic for every
$v\in\Omega$. By the local-global principle $(LG'_{r+2})$, $\psi'$ represents $a_1$ over $F$ whence a decomposition
$\psi'\cong \langle a_1\rangle\bot \psi'_1$. It then suffices to
apply $(LG_r)$ to the forms $\psi_1$ and $\psi'_1$, thanks to Witt's cancellation theorem.

Together with the argument in Remark$\;$\ref{remark4p13temp}, this
shows that if the natural map $\Br(F)\to\prod_{v\in\Omega}\Br(F_v)$
is injective and if the local-global principle with respect to
$\Omega$ holds for quadratic forms of rank $\ge 5$ over $F$, then
$(LG_r)$ is true for all $r\ge 2$. In particular, if $F$ is the function field of an algebraic curve over the fraction field of
a complete discrete valuation ring with \emph{arbitrary} residue field of characteristic $\neq 2$, then the analog of Thm.$\;$\ref{thm4p11temp} over $F$ is true by \cite[Theorems$\;$3.1 and 4.3]{CTPaSu}.  Note also that in this situation $(LG'_2)$ is false in general.
\end{remark}

\

\noindent \emph{Acknowledgements.} I got interested in the problems considered in this paper when I was
attending the workshop ``Deformation theory, patching, quadratic forms, and the Brauer group'' held at
American Institute of Mathematics (AIM) (Palo Alto, USA, January, 2011). I would like to thank AIM and the organizers
of this workshop for kind hospitality and generous support. I'm grateful to my advisor,
Prof.\! Jean-Louis Colliot-Th\'{e}l\`{e}ne, for helpful discussions and comments. Thanks are also due to Prof.\! Shuji Saito,
with whom a conversation has helped me to find the answer to a question that is needed in the paper. I also thank Professors Raman Parimala and Venapally Suresh
for sending me their new preprint on degree three cohomology. I'm indebted to the referee for a long list of valuable comments.

%\addcontentsline{toc}{section}{\textbf{References}}

\end{document}